\documentclass[10pt,a4paper]{amsart}
\usepackage[utf8]{inputenc}
\usepackage{amscd,amssymb,amsopn,amsmath,amsthm,graphics,amsfonts,enumerate,verbatim,calc}
\usepackage[dvips]{graphicx}
\usepackage[all]{xypic}
\usepackage{rotating}

\usepackage{mathpazo}
\usepackage{color}
\usepackage{latexsym}
\usepackage{amsthm,amsfonts,amssymb,mathrsfs}
\usepackage{rotating}
\usepackage[leqno]{amsmath}
\usepackage{xspace}
\usepackage[all]{xy}
\usepackage{longtable}
\headsep=5 mm \numberwithin{equation}{section}
\newtheorem{theorem}{Theorem}[section]
\newtheorem{lemma}[theorem]{Lemma}

\newtheorem{proposition}[theorem]{Proposition}
\newtheorem{corollary}[theorem]{Corollary}

\theoremstyle{definition}

\theoremstyle{remark}
\newtheorem{remark}[theorem]{Remark}
\newtheorem{fact}[theorem]{Fact}
\newtheorem{example}[theorem]{Example}
\newtheorem{observation}[theorem]{Observation}

\newtheorem{question}[theorem]{Question}
\newtheorem{conjecture}[theorem]{Conjecture}

\newtheorem{acknowledgement}{Acknowledgement}

\newcommand{\HH}{\operatorname{H}}

\newcommand{\Ass}{\operatorname{Ass}}

\newcommand{\im}{\operatorname{im}}

\newcommand{\grade}{\operatorname{grade}}

\newcommand{\Spec}{\operatorname{Spec}}

\newcommand{\tor}{\operatorname{tor}}
\newcommand{\pd}{\operatorname{pd}}

\newcommand{\Syz}{\operatorname{Syz}}

\newcommand{\rank}{\operatorname{rank}}

\newcommand{\V}{\operatorname{V}}

\newcommand{\Ext}{\operatorname{Ext}}
\newcommand{\Se}{\operatorname{S}}

\newcommand{\DVR}{\operatorname{DVR}}
\newcommand{\Supp}{\operatorname{Supp}}

\newcommand{\Tor}{\operatorname{Tor}}

\newcommand{\Hom}{\operatorname{Hom}}

\newcommand{\depth}{\operatorname{depth}}

\newcommand{\hdeg}{\operatorname{hdeg}}
\newcommand{\coker}{\operatorname{coker}}

\newcommand{\h}{h}

\newcommand{\lo}{\longrightarrow}
\newcommand{\fm}{\frak{m}}
\newcommand{\fp}{\frak{p}}

\newcommand{\fa}{\frak{a}}

\newcommand{\fn}{\frak{n}}

\begin{document}

\author[]{Mohsen Asgharzadeh}

\address{}
\email{mohsenasgharzadeh@gmail.com}

\title[ ]
{Finite support of tensor products}

\dedicatory{}
\begin{abstract}
	We determine the submodule of finite support of the tensor product of two modules $M$ and $N$ over a local ring, and estimate its length in terms of the properties of $M$ and $N$. In addition, we compute higher local cohomology modules of tensor products in a series of nontrivial cases. As applications, we calculate the depth of tensor powers and establish several criteria for freeness.
\end{abstract}

\maketitle

\tableofcontents

\section{Introduction}

Let $(R, \fm, k)$ be a commutative Noetherian local ring of dimension $d$, and assume all modules are finitely generated. Denote by $\HH^0_{\fm}(M)$ the submodule of $M$ consisting of elements annihilated by some power of $\fm$. For modules $M$ and $N$, we consider the local cohomology $\HH^0_{\fm}(M \otimes_R N)$, and let $\h^0(M \otimes_R N) := \ell(\HH^0_{\fm}(M \otimes_R N))$ denote its length.

\begin{question}[cf. {\cite[Page 704]{Wolmer}}]
	Can one estimate $\h^0(M \otimes_R N)$ in terms of $M$ and $N$?
\end{question}

Under various assumptions on the ring and the modules, Vasconcelos established several bounds on $\h^0(M \otimes_R N)$. For example, he considered the case where $R$ is regular and $N$ is locally free. He also raised the problem of extending these results to situations where $R$ is Gorenstein with an isolated singularity; see \cite[Question 8.2]{2010}. In \S2, we extend some of Vasconcelos' results and additionally address the singular case; see Propositions~\ref{v5} and~\ref{buchs}.

In particular, suppose that $R$ is Gorenstein with $d \geq 1$, and that $M$ has a presentation
\[
0 \to R^n \xrightarrow{\varphi} R^{n+d-1} \to M \to 0,
\]
where $I_n(\varphi)$ is $\fm$-primary. Then Vasconcelos proved the following inequality:
\[
\h^0(M \otimes_R M) \leq d \left((d-1)\deg(M) + \ell\left(\frac{R}{I_n(\varphi)}\right)\right)^2. \tag{$\ast$}
\]
In \cite[Question 8.1]{2010}, he asked how sharp this bound is compared to the actual value of $\h^0(M \otimes_R M)$. In \S3, we provide explicit computations to compare both sides of the inequality. In particular, we construct examples where the left-hand side of $(\ast)$ dominates $\h^0(M \otimes_R M)^2$; see Proposition~\ref{squar}.

As an alternative approach to \cite[Question 8.1]{2010}, we consider criteria for vanishing of the left-hand side of $(\ast)$:

\begin{proposition}
	Let $(R,\fm,k)$ be a local ring and $\fa \subseteq R$ an ideal. Suppose $\pd(M) < \infty$, and that at least one of the modules $M$ or $N$ is locally free on $\Spec(R) \setminus \V(\fa)$. Let $0 \leq r < d := \dim R$, and assume
	\[
	\grade_R(\fa, M) + \grade_R(\fa, N) \geq d + r + 1.
	\]
	Then $\HH^i_{\fa}(M \otimes_R N) = 0$ for all $0 \leq i \leq r$.
\end{proposition}

The special case $\fa = \fm$ was proven implicitly by Auslander and explicitly by Huneke and Wiegand for hypersurfaces. However, our results apply to more general settings.

In \S4, we partially address Vasconcelos' question about torsion in tensor products. For instance, if $(R, \fm)$ is a 3-dimensional Cohen--Macaulay local ring and $M$ is a reflexive module with $\pd(M) < \infty$, we prove that if $M^{\otimes 3}$ is torsion-free, then $M$ is free. In Corollary~\ref{4n}, we generalize this result to higher dimensions, confirming Vasconcelos' prediction.

In \S5, we examine the higher local cohomology modules $\HH^+_{\fm}(M \otimes_R N)$, dividing the discussion into four subsections. \S5.1 focuses on low-dimensional cases. \S5.2 provides explicit computations of $\ell(\HH^i_{\fm}(M \otimes_R M^*))$ and $\ell(\HH^i_{\fm}(M \otimes_R M))$ when $R$ is regular (see Proposition~\ref{main5b}). As an application, this yields a new negative answer to \cite[Question 3.5]{yosh2}. In \S5.3, we extend results of Auslander from regular rings to hypersurfaces.

Assume $R$ is Cohen--Macaulay of type at most two. By \cite[Theorem 6.1.2]{hsv}, the vanishing $\Ext^1_R(\omega_R, R) = \Ext^2_R(\omega_R, R) = 0$ implies that $R$ is Gorenstein. We offer a complementary result:

\begin{corollary}
	Let $R$ be a generically Gorenstein (e.g., reduced) Cohen--Macaulay local ring with a canonical module. If $R$ is of type at most two, then $\Ext^1_R(\omega_R, R) = 0$ if and only if $R$ is Gorenstein.
\end{corollary}

In addition, Proposition~\ref{vector2} supports Yoshida's conjecture \cite[Conjecture 3.4]{yosh}. We also prove:

\begin{proposition}
	Let $(R, \fm)$ be a Cohen--Macaulay local ring, $M$ a perfect module, and $N$ a locally free module of constant rank. Then for all $i < \dim(M)$,
	\[
	\h^i(M \otimes_R N) \leq \sum_{j = 0}^{\pd(M)} \beta_j(M) \h^{i+j}(N).
	\]
\end{proposition}

In \S5.4, we provide two criteria for freeness. First, suppose $R$ is regular and $M$ is locally free on $\Spec(R) \setminus \V(\fa)$, satisfying Serre's condition $(\Se_r)$. If $\HH^r_{\fa}(M^{\otimes (d - r)}) = 0$, then $M$ is free. This generalizes results of Auslander (for $r = 0$) and Huneke--Wiegand (for $1 \leq r \leq 2$ and $\fa = \fm$). The second criterion follows from Proposition 1.2:

\begin{corollary}\label{og}
	Let $R$ be a local ring. If $M$ is locally free over $\Spec(R) \setminus \V(\fa)$ and has finite projective dimension, then $M$ is free provided
	\[
	\grade(\fa, M) + \grade(\fa, M^*) \geq d + 2.
	\]
\end{corollary}

This generalizes a result of Hartshorne--Ogus, who assumed $R$ Gorenstein and $\fa = \fm$. As an application of Proposition 1.2, we observe:

\begin{observation}
	Let $(R, \fm)$ be a local ring of dimension $d$, and let $M$ be locally free on $\Spec(R) \setminus \V(\fa)$. Then
	\[
	\grade(\fa, M^{\otimes i}) \geq d - i \cdot \pd(M) \quad \text{for all } i > 1.
	\]
\end{observation}

This suggests studying the sequence $a_n := \depth(M^{\otimes n})$. What can be said about the asymptotic behavior of $(a_n)$?
In \S6, we compute depths of tensor powers and prove their stability in certain cases. For instance, the following result removes the regularity assumption from a theorem of Huneke--Wiegand:

\begin{proposition}\label{dpoc}
	Let $R$ be a local ring and $M$ a locally free module of projective dimension $1$. Then
	\[
	\depth(M^{\otimes i}) = \max\{0, \depth(R) - i\}.
	\]
\end{proposition}

We conclude by mentioning that a forthcoming work \cite{acs} will explore the asymptotic behavior of $\depth(\Syz_j(k)^{\otimes i})$ over rings of positive depth.

\section{Bounds on $\h^0(-\otimes \sim)$: after Vasconcelos}

By $\mu(-)$  we mean the minimal number of elements that need to generate $(-)$.

\begin{lemma}\label{0}
Let  $M$ be of finite length. Then
$\h^0(M\otimes_RN)\leq\ell(M)\mu(N)$.
\end{lemma}

\begin{proof}
	The proof is by induction on $\ell(M)$. Suppose $\ell(M)=1$. Then
 $M=R/ \fm$. By definition, $\HH^0_{\fm}(M\otimes_RN)=M\otimes_RN=\frac{N}{\fm N}$
and so $\h^0(M\otimes_RN)=\mu(N)=\ell(M)\mu(N)$. We look at the exact sequence
$0\to R/ \fm\to M\to \overline{M}\to 0$ where $\ell(\overline{M})=\ell(M)-1$.
By induction, $\ell(\overline{M}\otimes_R N)\leq\ell(\overline{M})\mu(N)$. The sequence induces
$ R/ \fm\otimes_RN\stackrel{g}\lo M\otimes_RN\stackrel{f}\lo \overline{M}\otimes_RN\to 0$.
Since $R/ \fm\otimes_RN\twoheadrightarrow \im(g)\to 0$ is surjective,  $\ell(\ker(f))=\ell(\im(g))\leq\mu(N)$. We have
$$ \ell(M\otimes_R N)=\ell(\overline{M}\otimes_R N)+\ell(\ker(f))\leq\ell(\overline{M}\otimes_RN)+\ell(N/ \fm
N)\leq\ell(\overline{M})\mu(N)+\mu(N).$$
So, $\ell(\HH^0_{\fm}(M\otimes _RN))= \ell(M\otimes _RN)\leq(\ell(M)-1)\mu(N)+\mu(N)=\mu(N)\ell(M)$.
\end{proof}

The particular case of the next result stated in \cite[Proposition 2.1]{2010} without a proof:

\begin{lemma}\label{reduction} 
	One has $\h^0(M\otimes_RN)\leq\h^0(M)\mu(N)+\h^0(\frac{M}{\HH^0_{\fm}(M)}\otimes_RN)$.
In particular,  $$\h^0(M\otimes_RN)\leq\h^0(M)\mu(N)+\h^0(N)\mu(N)+\h^0(M/ \HH^0_{\fm}(M) \otimes_R\ N/ \HH^0_{\fm}(N)).$$
\end{lemma}

\begin{proof}We may assume neither $M$ nor $N$ are of finite length (see Lemma \ref{0}). We look at
$0\to \HH^0_{\fm}(M)\to M\to \widetilde{M}:=\frac{M}{\HH^0_{\fm}(M)}\to 0$.
Apply $-\otimes_RN$ to it and look at the induced long exact sequence
$$\Tor_1^R(\widetilde{M},N)\to \HH^0_{\fm}(M)\otimes_RN\stackrel{f}\lo M\otimes_RN\to \widetilde{M}\otimes_RN\to 0.$$
The sequences $0\to \ker(f)\to M\otimes_RN\to \widetilde{M}\otimes_RN\to 0 $ and
 $\Tor_1^R(\widetilde{M},N)\to \HH^0_{\fm}(M)\otimes_RN\to \ker(f)\to 0 $ are exact.
 From the second,  $\ell(\ker(f))\leq\ell(\HH^0_{\fm}(M)\otimes_RN)\leq\h^0(M)\mu(N)$, see Lemma \ref{0}.
The first one deduces the  exact sequence
$0\to  \HH^0_{\fm}(\ker(f))\to  \HH^0_{\fm}( M\otimes_RN)\to \HH^0_{\fm}(\widetilde{M}\otimes_RN)\to \HH^1_{\fm}(\ker(f)).$
So,
$\h^0(M\otimes_R N)\leq \h^0(\ker(f))+ \h^0(\widetilde{M}\otimes _RN) = \ell(\ker(f))+ \h^0(\widetilde{M}\otimes _RN) \leq\h^0(M)\mu(N)+\h^0(\widetilde{M}\otimes_R N).$
Repeat this for $N$, we have \[\begin{array}{ll}
\h^0(M\otimes_RN)&\leq\h^0(M)\mu(N)+\h^0(\widetilde{M}\otimes_RN)\\
&\leq\h^0(M)\mu(N)+\h^0(N)\mu(\widetilde{N})+\h^0(\widetilde{M}\otimes_R\widetilde{N})\\
&\stackrel{(\ast)}\leq\h^0(M)\mu(N)+\h^0(N)\mu(N)+\h^0(\widetilde{M}\otimes_R\widetilde{N}),
\end{array}\]where $(\ast)$ follows by applying $(-)\otimes_R R/\fm$ to $N\twoheadrightarrow \widetilde{N}\to 0$
to see that $N/ \fm N\twoheadrightarrow \widetilde{N}/ \fm \widetilde{N}\to 0$. In particular, $\dim(\widetilde{N}/ \fm \widetilde{N})\leq \dim(N/ \fm N)$. This completes the proof.
\end{proof}

By $\h^i(-)$  we mean $\ell(\HH^i_{\fm}(-))$  provided it is finite.
By  $\pd(-)$ we mean  the projective dimension.
We look at the minimal free resolution of $M$:
 $\cdots \to R^{\beta_{i}(M)}\stackrel{f_{i}}\lo  R^{\beta_{i-1}(M)}\to\cdots\to  R^{\beta_{0}(M)}\to M\to 0.$
The $i^{th}$ \textit{syzygy} module of $M$ is $\Syz_i(M) := \ker(f_{i-1})$ for all $i>0$.
The following is in  \cite[Theorem 4.1]{2010} under the  additional assumption that  $R$ is Gorenstein.

\begin{proposition}\label{vector}
Let $R$ be an equi-dimensional and  generalized Cohen-Macaulay local ring,  and $N$ be locally free and of constant rank over the punctured spectrum. If $\pd(M)<\depth
(R)$, then
$\h^0(M\otimes_RN)\leq
\sum_{i=0}^{\pd(M)}\beta_i(M)\h^i(N)$.
\end{proposition}

\begin{proof}
Let $p:=\pd(M)$. We may assume  $N$ is not of finite length (see Lemma \ref{0}). The assumptions implies that $N$ is generalized Cohen-Macaulay and of dimension equal to $\dim (R)$.
We look at $0\to \Syz_1(M)\to R^{\beta_{0}(M)}\to M\to 0$. Apply $-\otimes_RN$ to it and look at the induced long exact sequence
$$0\to\Tor_1^R(M,N)\to \Syz_1(M)\otimes_RN\stackrel{f}\lo R^{\beta_{0}(M)}\otimes_RN\to M\otimes_RN\to 0.$$
We have
$0\to \ker(f)\to R^{\beta_{0}}\otimes_R N\to M\otimes_R N\to 0 $ and
 $0\to\Tor_1^R(M,N)\to \Syz_1(M)\otimes_R N\to \ker(f)\to 0.$
Since
$N$ is locally free, $\Tor_1^R(M,N)$ is of finite length. Thus, $\HH^0_{\fm}(\Tor_1^R(M,N))=\Tor_1^R(M,N)$ and $\HH^1_{\fm}(\Tor_1^R(M,N))=0$.
We apply $\Gamma_{\fm}$ to these sequences to deduce the following:\begin{enumerate}
\item[]  $0\to\HH^0_{\fm}(\Tor_1^R(M,N))\to \HH^0_{\fm}(\Syz_1(M)\otimes_R N)\to\HH^0_{\fm}(\ker(f))\to\HH^1_{\fm}(\Tor_1^R(M,N))=0, $
\item[]  $0\to \HH^0_{\fm}(\ker(f))\to \HH^0_{\fm}(R^{\beta_{0}(M)}\otimes_RN)\to \HH^0_{\fm}(M\otimes_RN)\to \HH^1_{\fm}(\ker(f)).$
\end{enumerate}
Also,
$\HH^+_{\fm}(\Syz_1(M)\otimes_RN)\simeq\HH^+_{\fm}(\ker(f))$. We use these to conclude that:
$$
\h^0(M\otimes_RN)\leq \ell(\HH^1_{\fm}(\ker(f)))+\beta_{0}(M)\h^0(N)=\ell(\HH^1_{\fm}(\Syz_1(M)\otimes_RN))+\beta_{0}(M)\h^0(N).
$$
In the same vein, $\ell(\HH^1_{\fm}(\Syz_1(M)\otimes_RN))\leq\ell(\HH^2_{\fm}(\Syz_2(M)\otimes_RN))+\beta_{1}(M)\h^1(N).$
Thus
\[\begin{array}{ll}
\h^0(M\otimes_RN)&\leq \ell(\HH^1_{\fm}(\Syz_1(M)\otimes_RN))+\beta_{0}(M)\h^0(N)\\
&\leq\ell(\HH^2_{\fm}(\Syz_2(M)\otimes_RN))+\beta_{1}(M)\h^1(N)+\beta_{0}(M)\h^0(N).
\end{array}\]  Repeating this, $
\h^0(M\otimes_RN)
\leq\ell(\HH^p_{\fm}(\Syz_{p}(M)\otimes_RN))+\sum_{i=0}^{p-1}\beta_i(M)\h^i(N)=\sum_{i=0}^{p}\beta_i(M)\h^i(N).$
\end{proof}
By $\hdeg(M)$ we mean the \textit{cohomological degree}, see \cite{Wolmer} for its definition.
The following  contains more data than \cite[Theorem 4.2]{2010} via dealing with $\pd(A)=\dim(R)$.

\begin{proposition}\label{cvector}
Let $R$ be a $d$-dimensional regular local ring, $M$ a module  and $N$ be locally free  over the punctured spectrum. Then\begin{equation*}
\h^0(M\otimes_RN)\leq \left\{
\begin{array}{rl}
d\hdeg(M)\hdeg(N) & \  \   \   \   \   \ \  \   \   \   \   \ \text{if }  \pd(M)< d\\
(d+1)\hdeg(M)\hdeg(N)-1 & \  \   \   \   \   \ \  \   \   \   \   \ \text{if } \pd(M)= d
\end{array} \right.
\end{equation*}
\end{proposition}

\begin{proof}Due to Lemma \ref{0} we can assume that neither $M$ nor $N$ are artinian. The claim in the case  $\pd(M)< d$
is in \cite[Theorem 4.2]{2010}. Suppose $\pd(M)=d$. Since $M$ is not artinian, $M\neq \Gamma_{\fm}(M)$.
We denote $M/ \Gamma_{\fm}(M)$ by $\widetilde{M}$. Note that $\depth(\widetilde{M})>0$. Due to Auslander-Buchsbaum formula,
$\pd(\widetilde{M})< d$. We combine Lemma \ref{reduction} with the first part to see
$$\h^0(M\otimes_RN)\leq\h^0(M)\mu(N)+\h^0(\widetilde{M}\otimes_RN)\leq\h^0(M)\mu(N)+d\hdeg(\widetilde{M})\hdeg(N).$$
Recall from definition that  $\h^0(M)\leq\hdeg(M)$. By \cite[Theorem 1.10]{Wolmer}, $\beta_i(N)\leq\beta_i(k)\hdeg(N)$.
We use this for $i=0$ to see
$\mu(N)\leq\hdeg(N)$. In view of \cite[Proposition 2.8(a)]{Wolmer} we have $\hdeg(\widetilde{M})=\hdeg(M)-\ell(\Gamma_{\fm}(M) )<\hdeg(M)$. We
putt all of these together to see $$
\h^0(M\otimes_RN)\leq\h^0(M)\mu(N)+d\hdeg(\widetilde{M})\hdeg(N)<\hdeg(M)\hdeg(N)+d \hdeg(M) \hdeg(N).$$
The claim is now clear.
\end{proof}

\begin{corollary}\label{1}
Let $R$ be a $d$-dimensional regular local ring. Assume one of the following items hold:
 i) $d=1$,
ii) $d=2$ and $M$ is torsion-free,
iii) $d=3$ and $M$ is  reflexive.
Then $\h^0(M\otimes_RN)<(d+1)\hdeg(M)\hdeg(N)$ for any  finitely generated module $N$.
\end{corollary}

\begin{proof}
It follows that $M$  is locally free. In view of Proposition \ref{cvector} we get the desired claim.
\end{proof}

The  next result slightly extends \cite[Proposition 3.4]{2010}:

\begin{corollary}\label{1g}
Let $(R,\fm)$ be a $1$-dimensional complete local integral domain containing a field,   $M$ and $N$ be finitely generated. Let $J$ be the Jacobian ideal. Then
$$\h^0(M\otimes_RN)\leq\hdeg(M)\hdeg(N)(2+\deg(R)\ell(\frac{R}{J}))-\rank(M)\rank(N)\deg(R)\ell(\frac{R}{J}).$$In particular,
$\h^0(M\otimes_RN)\leq(2+\deg(R)\ell(\frac{R}{J}))\hdeg(M)\hdeg(N)$.
\end{corollary}

\begin{proof}Due to Lemma \ref{0}, we may assume that neither $M$ nor $N$ are artinian. Let $\widetilde{M}:=\frac{M}{\HH^0_{\fm}(M)}$.
This is nonzero and of positive depth. Thus, $\widetilde{M}$ is maximal Cohen-Macaulay. Over any 1-dimensional  reduced local ring, the
category of maximal Cohen-Macaulay modules coincides with the category of torsion free modules. Hence $\widetilde{M}$ and  $\widetilde{N}$ are torsion free.
In view of \cite{i}, we see $J\Ext^2_R(-,\sim)=0$. We combine this with the proof of  \cite[Proposition 3.4]{2010}
to see $\h^0(\widetilde{M}\otimes_R\widetilde{N})\leq\left(\mu(\widetilde{M})\mu(\widetilde{N})-\rank(\widetilde{M})\rank(\widetilde{N})\right)\deg(R)\ell(\frac{R}{J}).$
Recall that $\mu(\widetilde{M})\leq \mu(M)$. Denote the fraction field of $R$ by $Q(R)$. Recall that $\HH^0_{\fm}(M)\otimes_R Q(R)=0$. We apply the exact functor $-\otimes_RQ(R)$ to
$0\to\HH^0_{\fm}(M)\to M \to \widetilde{M}\to 0$ to see the sequence  $0=\HH^0_{\fm}(M)\otimes_R Q(R)\to M\otimes_R Q(R)\to \widetilde{M}\otimes_R Q(R)\to 0$ is exact. From this
$\rank(M)=\rank(\widetilde{M})$. Therefore, $$\h^0(\widetilde{M}\otimes_R\widetilde{N})
\leq\left(\mu(M)\mu(N)-\rank(M)\rank(N)\right)\deg(R)\ell(\frac{R}{J}).$$
In view of Lemma \ref{reduction} we have
 \[\begin{array}{ll}
\h^0(M\otimes_RN)&\leq\h^0(M)\mu(N)+\h^0(N)\mu(N)+\h^0(\widetilde{M}\otimes_R\widetilde{N})\\
&\leq \h^0(M)\mu(N)+\h^0(N)\mu(N)+\left(\mu(M)\mu(N)-\rank(M)\rank(N)\right)\deg(R) \ell(\frac{R}{J})\\
&\leq\hdeg(M)\hdeg(N)(2+\deg(R) \ell(\frac{R}{J}))-\rank(M)\rank(N)\deg(R) \ell(\frac{R}{J}).
\end{array}\]
\end{proof}

Here, the notation $M^\ast$ stands for $\Hom_R(M,R)$.
\begin{proposition}\label{v5}
Let $R$ be a  Gorenstein ring with isolated singularity and $M$ be  maximal Cohen-Macaulay. Then
$\h^0(M\otimes_RN)$ can  estimate in terms of $M$ and $N$.
\end{proposition}

\begin{proof}Maximal Cohen-Macaulay modules over Gorenstein rings are reflexive, e.g., $M$ is reflexive. We may assume  $N$ is not of finite length (see Lemma \ref{0}). In view of Lemma
\ref{reduction},
we may replace $N$ with $N/ \Gamma_{\fm}(N)$ and assume in addition that $\depth(N)>0$. This implies that
$\Hom_R(-,N)$ has positive depth provided $\Hom_R(-,N)\neq 0$. Let $D(-)$ be the Auslander's transpose. We look at the exact sequence
$$\Tor^R_2 (D(M^{\ast}),N)\stackrel{f}\lo M^{\ast\ast}\otimes _RN\stackrel{g}\lo\Hom_R(M^{\ast},N)\stackrel{h}\lo\Tor^R_1 (D(M^{\ast}),N)\to 0.$$
Without loss of the generality we can assume that $\Hom_R(-,N)\neq 0$. Note that $M^{\ast}$ is   maximal Cohen-Macaulay and so locally free over punctured spectrum. Since $D(-)$ behaves nicely
with respect to localization,
we see that $D(M^{\ast})$ is of finite length. Hence $\Tor^R_2 (D(M^{\ast}),N)$ is of finite length. Due to $\Tor^R_2 (D(M^{\ast}),N)\twoheadrightarrow \im(f)\to 0$
we see $\im(f)$ is of finite length. We have the following
exact sequences  $0\to\ker(h)\to\Hom_R(M^{\ast},N)\to\Tor^R_1 (D(M^{\ast}),N)\to 0 $ and
$0\to\ker(g)\to M^{\ast\ast}\otimes_R N\to\ker(h)\to 0.$ Also,
$\Tor^R_2 (D(M^{\ast}),N)\twoheadrightarrow\im(f)=\ker(g).$
Since $\depth(\Hom(M^{\ast},N))>0$ the first sequence says that $\depth(\ker(h))>0$.
From the second sequence we have $\h^0(M\otimes_RN)=\h^0(\ker(g)).$ From the third, we have
$\h^0(\ker(g))=\ell(\im(f))\leq\ell(\Tor^R_2 (D(M^{\ast}),N))$. In sum,
$$\h^0(M\otimes_RN)\leq\ell(\Tor^R_2 (D(M^{\ast}),N))\leq \beta_2(N)\ell(D(M^{\ast})),$$as claimed
\end{proof}

\begin{proposition}\label{buchs}
Let $(R,\fm)$ be a Cohen-Macaulay local ring of dimension $d>1$, $M$ be perfect of projective dimension one  and $N$ be Buchsbaum
of dimension $d$. Then
$\h^0(M\otimes_RN)<
3\hdeg(M)\hdeg(N)$. Suppose in addition that
$\depth(N)>0$. Then $\h^0(M\otimes_RN)\leq
2\hdeg(M)\hdeg(N)$.
\end{proposition}

\begin{proof}
 Let $\widetilde{N}:=\frac{N}{\HH^0_{\fm}(N)}$. In view of \cite[Proposition I.2.22]{bus}, $\widetilde{N}$ is Buchsbaum. Since $\dim(N)=d>0$, we deuce that $\widetilde{N}\neq 0$.
It follows by definition that $\depth(\widetilde{N})>0$,  $\HH^+_{\fm}(\widetilde{N})\simeq\HH^+_{\fm}(N)$ and
that $\dim(N)=\dim(\widetilde{N})$. Recall from \cite[Proposition 2.7]{yosh}:
\begin{enumerate}
\item[Fact A)]  Let $A$ be a Cohen-Macaulay local ring of dimension $d>1$ and  $P$ be perfect of depth one.
If $Q$ is  Buchsbaum of positive depth and maximal dimension,
then $\h^0(P\otimes_A Q)= \mu(P)(\h^0(Q) + \h^1(Q))$.
\end{enumerate} Recall that $\hdeg(\widetilde{N})=\hdeg(N)-\ell(\Gamma_{\fm}(N) )$, $\mu(-)\leq\hdeg(-)$ and that $\h^{<d}(-)\leq\hdeg(-)$.
In view of Lemma \ref{reduction} we have
 \[\begin{array}{ll}
\h^0(M\otimes_RN)&\leq\h^0(N)\mu(M)+\h^0(M\otimes_R\widetilde{N})\\
&= \h^0(N)\mu(M)+\mu(M)(\h^0(\widetilde{N}) + \h^1(\widetilde{N}))\\
&\leq\hdeg(M)\hdeg(N)+2\hdeg(M)\hdeg(\widetilde{N})\\
&=\hdeg(M)\hdeg(N)+2\hdeg(M)(\hdeg(N)-\Gamma_{\fm}(N) )\\
&\leq3\hdeg(M)\hdeg(N),
\end{array}\]and we remark that if $\Gamma_{\fm}(N)\neq 0$, then the last inequality
is strict. This completes the proof.
\end{proof}

Let $R$ be a 3-dimensional regular local ring, $M$ and $N$ be torsionfree. Theorem 6.1 in \cite{2010} says that
$\h^0(M\otimes_RN)<4 \hdeg(M) \hdeg(N)$. I feel that its   proof  says:

\begin{fact}\label{3ht}
Let  $(R,\fm)$  be a 3-dimensional regular local ring, $M$ and $N$ be torsionfree. Then
$\h^0(M\otimes_RN)<16 \hdeg(M)  \hdeg(N).$
\end{fact}

\begin{proof}
Let $C:=\coker(M\to M^{\ast\ast})$.
Vasconcelos proved that\begin{enumerate}
		\item[i) ] $\h^0(M\otimes_RN)\leq 3\hdeg(M^{**})  \hdeg(N) +\h^0(\Tor_1(M,N))$
		\item[ii) ] $\h^0(\Tor_1(M,N))\leq\h^0(\Syz_1(N)\otimes_RC) $
\item[ iii)] $\hdeg(\Syz_1(N))<4\hdeg(N)$
\item[ iv)]$\hdeg(M)=\hdeg(M^{**})+\hdeg(C)$.
	\end{enumerate}
We put things together to see that
 \[\begin{array}{ll}
\h^0(M\otimes_RN)&\leq  3\hdeg(M^{**})  \hdeg(N) +\h^0(\Tor_1(M,N)) \ \ \ \ \  \ \ \ \ \ \ \ \ \ \   \ \ \ \ \  \ \ \ \ \ \ \ \ \ \  \ \ \ \ \ \ \  \ \ \ \ \  (i) \\
&\leq  3\hdeg(M^{**})  \hdeg(N) + \h^0(\Syz_1(N)\otimes_RC)\ \ \  \ \ \ \ \ \ \ \  \ \ \  \ \ \  \ \ \ \ \  \ \ \ \ \ \ \ \ \ \   \ \ \ \ \      (ii)\\
&<  3\hdeg(M^{**})  \hdeg(N) + 4\hdeg(\Syz_1(N))\hdeg(C)  \ \ \  \ \ \ \ \  \ \ \ \ \ \ \ \ \ \   \ \ \ \ \      \ref{cvector}\\
&<  3\hdeg(M^{**})  \hdeg(N)+16\hdeg(N)\hdeg(C) \ \ \ \ \ \ \ \ \ \  \ \ \ \ \ \ \ \ \ \  \ \ \ \ \ \ \ \ \ \   \ \   (iii)\\
&<  16\hdeg(M^{**})  \hdeg(N)+16\hdeg(N)\hdeg(C)\\
&=  16\hdeg(M) \hdeg(N)\ \ \ \ \  \ \ \ \ \ \ \ \ \ \   \ \ \ \ \  \ \ \ \ \ \ \ \ \ \  \ \ \ \ \ \ \ \ \ \   \ \ \ \ \  \ \ \ \ \ \ \ \ \ \  \ \ \ \ \ \ \ \ \ \   \ \ \ \ \  \ \ \ \ \ (iv),
\end{array}\]as claimed.
\end{proof}

\begin{remark}
	Suppose	\( \pd(B) = 1 \).
	Then
	\[
	h^0(A \otimes B) \leq h^0\!\left( \frac{A}{\tor(A)} \otimes B \right) + h^0\!\left( {\tor(A)} \otimes B \right).
	\]In particular,  	$h^0(- \otimes B) $ is reduced to compute over torsion and torsion-free modules.
\end{remark}

\begin{proof}
Apply \( -\otimes B \) to \( 0 \to \tor(A) \to A \to \frac{A}{\tor(A)} \to 0 \). This gives
	\[
	\Tor_1(\frac{A}{\tor(A)}, B) \to \tor(A) \otimes B \to A \otimes B \to \frac{A}{\tor(A)} \otimes B \to 0.
	\]
But \( \frac{A}{\tor(A)} \) is torsion-free, so there exists a free module $F$ and the following exact sequence \[0\to \frac{A}{\tor(A)}\to F\to \Omega^{-1} (\frac{A}{\tor(A)} ) \to 0, \] 
	so \( \Tor_1^R(\frac{A}{\tor(A)} , B) = \Tor_2^R(\Omega^{-1}(  \frac{A}{\tor(A)}), B) = 0 \), as
	\(
	\pd(B) \leq 1.
	\) Applying the local cohomology gives us
	\[
	0 \to H^0_{\fm}(\tor(A) \otimes B) \to H^0_{\fm}(A \otimes B) \to H^0_{\fm}(\frac{A}{\tor(A)} \otimes B).   
	\] This gives 
	\[
	h^0(A \otimes B) \leq h^0\!\left( \frac{A}{\tor(A)} \otimes B \right) + h^0\!\left( {\tor(A)} \otimes B \right).
	\]\end{proof}

\begin{remark}
	We adopt the following assumptions:
	\begin{enumerate}
		\item $R$ is $3$-dimensional and normal.
		\item $A$ is torsion and $H^1_{\mathfrak{m}}(A)$ is finite.
		\item $B$ is torsion-free and $\operatorname{pd}(B) = 1$.
	\end{enumerate}
	Then
	\[
	h^0(A \otimes B) \;\leq\; h^0\!\bigl( A^{\beta_0(B)} \bigr) \;+\; h^1\!\bigl( A^{\beta_1(B)} \bigr).
	\]
\end{remark}

\begin{proof}
	Set $\beta_0(B) = n$ and $\beta_1(B) = m$.
	From the exact sequence \(0 \to R^m \to R^n \to B \to 0\) we obtain
	\[
	0 \to \operatorname{Tor}^R_1(A, B) \to A^m \to A^n \to A \otimes B \to 0.
	\]
	This splits into two short exact sequences:
	\begin{enumerate}
		\item \(0 \to \operatorname{Tor}^R_1(A, B) \to A^m \to C \to 0\);
		\item \(0 \to C \to A^n \to A \otimes B \to 0\).
	\end{enumerate}
	From (1) we obtain the local cohomology exact sequence
	\[
	H^1_{\mathfrak{m}}(A^m) \to H^1_{\mathfrak{m}}(C) \to H^2_{\mathfrak{m}}(\operatorname{Tor}^R_1(A, B)).
	\]
	From (2) we obtain
	\[
	H^0_{\mathfrak{m}}(A^n) \to H^0_{\mathfrak{m}}(A \otimes B) \to H^1_{\mathfrak{m}}(C).
	\]
	Hence,
	\[
	h^0(A \otimes B) \;\leq\; h^0(A^n) + h^1(C).
	\]
	
	\paragraph{Claim:} \(\dim \operatorname{Tor}^R_1(A, B) \leq 1\).
	
	\paragraph{Proof of the claim:}
	Suppose otherwise. Then there exists a prime ideal \(\mathfrak{p} \in \operatorname{Supp}(\operatorname{Tor}_1(A, B))\) such that \(\dim R/\mathfrak{p} \geq 2\). Consequently, \(\operatorname{ht}(\mathfrak{p}) = 1\).
	In this case, the localization \(R_{\mathfrak{p}}\) is a discrete valuation ring. Since \(B_{\mathfrak{p}}\) is torsion-free, it follows that \(B_{\mathfrak{p}}\) is free over \(R_{\mathfrak{p}}\).
	Therefore,
	\[
	\operatorname{Tor}^R_1(A, B)_{\mathfrak{p}} = \operatorname{Tor}^{R_{\mathfrak{p}}}_1(A_{\mathfrak{p}}, B_{\mathfrak{p}}) = 0,
	\]
	which contradicts the choice of \(\mathfrak{p}\) in the support of \(\operatorname{Tor}_1(A, B)\). This establishes the claim.
	
	Applying the claim to the earlier exact sequence and using Grothendieck's vanishing theorem, we obtain \(H^2_{\mathfrak{m}}(\operatorname{Tor}_1(A, B)) = 0\). Hence the sequence becomes
	\[
	H^1_{\mathfrak{m}}(A^m) \to H^1_{\mathfrak{m}}(C) \to 0,
	\]
	so that \(h^1(C) \leq h^1(A^m)\).
Combining this inequality with \(h^0(A \otimes B) \leq h^0(A^n) + h^1(C)\) yields the desired bound.
\end{proof}

\section{Toward sharpening the  bound  on $\h^0(M\otimes_R M)$}

We look at   $M$ with a presentation of the form
$0\to R^n\stackrel{\varphi}\lo R^{n+d-1} \to M\to 0$ where $d=\dim R$.  Recall that \cite[Question 8.1]{2010} deals with the sharpness of  $$\h^0 (M\otimes_RM)\leq d\left((d - 1)\deg(M) +
\ell(\frac{R}{I_n(\varphi)})\right)^2\quad(\ast)$$
Suppose $d=2$ and $n=1$. Let us repeat the assumption: $M$ has a presentation of the form
$0\to R\stackrel{\varphi}\lo R^{2} \to M\to 0$ where the ideal $I_1(\varphi)$
 is $\fm$-primary. Then the validity of bound $(\ast)$  simplifies to the validity of $$\h^0 (M\otimes_RM)\leq 2(\deg(M) + \ell(\frac{R}{I_1(\varphi)}))^2.$$ We start by looking at a situation for which $\ell(\frac{R}{I_1(\varphi)})$ is minimal:

\begin{example}\label{m20}
Let $(R,\fm,k)$ be a $2$-dimensional regular local ring. Then $\h^0(\fm\otimes_R\fm)=1$.
\end{example}

Note that $\fm$ has a presentation of the form
$0\to R\stackrel{\varphi}\lo R^{2} \to \fm\to 0$ where the ideal $I_1(\varphi)=\fm$.
\begin{proof}
Let
 $x$ and $y$ be a generating set of $\fm$ and look at $\zeta:=x\otimes y-y\otimes x$. We have $$x\zeta=x(x\otimes y-y\otimes x)=x^2\otimes y-xy\otimes x=xy\otimes x-xy\otimes x=0.$$
 Similarly,
$y\zeta=0$, so that $\fm\zeta=0$. By definition, $\zeta\in\HH^0_{\fm}(\fm\otimes_R\fm)$. Again due to definition, $\HH^0_{\fm}(\fm\otimes_R\fm)$ is submodule of the torsion part of $
\fm\otimes_R\fm  $. On the other hand,
the torsion part of $ \fm\otimes_R\fm  $ is $\Tor^R_2(k,k)$ (see \cite[Lemma 1.4]{tensor}) which is a vector space of dimension equal to $\beta_2(k)=1$. From these,
$\HH^0_{\fm}(\fm\otimes_R\fm)=\zeta R\simeq k$.
In particular, $\h^0(\fm\otimes_R\fm)=\ell(\HH^0_{\fm}(\fm\otimes_R\fm))=1$.
 \end{proof}

The difference $2(\deg(M) + \ell(\frac{R}{I_n(\varphi)}))^2-\h^0 (M\otimes_RM)$ may be large:

\begin{proposition}\label{squar}
Let $(R,\fm,k)$ be a $2$-dimensional Cohen-Macaulay local domain and $I$ be an ideal generated by a full parameter sequence. Then $\h^0(I\otimes_RI)=\hdeg(R/I)$. In particular, $$\h^0 (I\otimes_RI)=\ell(R/I)\lneqq
2\left(\deg(I) + \ell(R/I)\right)^2.$$
\end{proposition}
Note that $I$ has a presentation of the form
$0\to R\stackrel{\varphi}\lo R^{2} \to I\to 0$ where the ideal $I_1(\varphi) $ is $\fm$-primary.
\begin{proof}Let
 $x$ and $y$ be a generating set of $I$. The notation  $\mathbb{K}(I; R)$ stands for the Koszul complex of $R$ with respect to $I$. That is $$\mathbb{K}(I; R):=0\lo R\stackrel{(^{+y}_{-x})}\lo R^2\stackrel{(x,y)}\lo R\lo R/I\lo 0.$$ This is a minimal free
resolution of $R/I$. In view of definition,   $$\mathbb{K}(I; R)\otimes_R R/I\simeq0\lo R/I\stackrel{0}\lo R/I\oplus R/I\stackrel{0}\lo R/I\lo R/I\otimes
R/I\lo 0.$$
By definition,  $$\tor(I\otimes _RI)\simeq\Tor^R_2(R/I,R/I)\simeq\HH_2(\mathbb{K}(I; R)\otimes_R \frac{R}{I})\simeq\frac{R}{I}.$$
We look at the exact sequence $$0\to \tor(I\otimes _RI)\to I\otimes_R I\to \frac{I\otimes _RI}{\tor(I\otimes_R I)}\to 0.$$
Since $\frac{I\otimes_R I}{\tor(I\otimes _RI)}$ is torsion-free, $\HH^0_{\fm}(\frac{I\otimes_R I}{\tor(I\otimes_R I)})=0$. We put this in
$$0\to \HH^0_{\fm}( \tor(I\otimes _RI) )\to\HH^0_{\fm}( I\otimes _RI )\to \HH^0_{\fm}(\frac{I\otimes _RI}{\tor(I\otimes_R I)})$$ to see that $\HH^0_{\fm}( \tor(I\otimes _RI) )\simeq\HH^0_{\fm}( I\otimes_R I )$. Since $\ell(\frac{R}{I})<\infty$,  $$\HH^0_{\fm}( I\otimes _RI )\simeq\HH^0_{\fm}( \tor(I\otimes_R I) )\simeq\HH^0_{\fm}( R/I )\simeq R/I.$$
Thus, $\h^0(I\otimes_RI)=\ell(R/I)$.
 \end{proof}

In our 2-dimensional approach, $\h^0(M\otimes_RM)$ rarely vanishes:

\begin{observation}\label{2au}
Let $(R,\fm,k)$ be a $2$-dimensional regular local ring  and $0\neq M$ be torsion-free. Then $\h^0(M\otimes_RM)=0$ if and only if $M$ is free.
\end{observation}

\begin{proof}
The if part is trivial. Suppose $M$ is not free. Since $M$ is $(\Se_1)$ it follows that $\pd(M)=1$.
We claim that $\Tor_1^R(M,M)=0$. Suppose on the contradiction that  $\Tor_1^R(M,M)\neq0$. Let $\fp$ be any height
one prime ideal. Since $R_{\fp}$ is a discreet valuation ring and $M_{\fp}$ is torsion-free, it follows that $M_{\fp}$ is free over $R_{\fp}$.
From this, $\Tor_1^R(M,M)$ is of finite length. Thus, $\depth(\Tor_1^R(M,M))=0$.
We recall the following result of Auslander (see \cite[Theorem 1.2]{au}):
\begin{enumerate}
\item[Fact A)]Let $S$ be a local ring, $\pd(A)<\infty$. Let $q$ be the largest number such that  $\Tor_q^S(A, B)\neq0$. If $\depth(\Tor_q^S(A, B))\leq1$,
then $\depth(B)=\depth(\Tor_q^S(A, B))+\pd(A)-q.$
\end{enumerate}
We use this for $A=B=M$ and $q=1$, to see  $$1=\depth(M)=\depth(\Tor_1^R(M, M))+\pd(M)-q=0+1-1=0,$$ a contradiction. Thus,  $\Tor_1^R(M,M)=0$.
This vanishing result allow us to use:
\begin{enumerate}
\item[Fact B)] (see \cite[Corollary 1.3]{au}) Let $S$ be a local ring, $A$ and $B$ be of finite projective dimension. If $\Tor_+^S(A, B)=0$, then $\pd(A)+\pd(B)=
\pd(A\otimes_S B)$.
\end{enumerate}
From this, $\pd(M\otimes_RM)=2$. By Auslander-Buchsbaum, $\depth(M\otimes_RM)=0$. Consequently, $\h^0(M\otimes_RM)\neq0$.
\end{proof}

The above observation extends in the following sense:

\begin{proposition}\label{g} Let $(R,\fm,k)$ be a  local ring with an ideal $\fa$, $M$ and $N$  be  such that $\pd(M)<\infty$ and  one of them is locally free over
	$\Spec(R)\setminus\V(\fa)$. Let $0\leq r<d:=\dim R$ be such that
$\grade_R(\fa,M)+\grade_R(\fa,N)\geq d+r+1$. Then   $\HH_{\fa}^0(M\otimes_RN)=\ldots=\HH_{\fa}^r(M\otimes_RN)=0$.
\end{proposition}

\begin{proof}Without loss of the generality, neither $M=0$ nor $N=0$.
 We claim that grade of $\fa$ with respect to $N$ and $M$ is at least $r+1$. To this end
 recall that $d=\dim(R)\geq\dim(N)\geq\depth(N)\geq\grade_R(\fa,N)$. We put this into the assumption:
 $$\grade_R(\fa,M)+d\geq\grade_R(\fa,M)+\grade_R(\fa,N)\geq d+r+1,$$i.e., $\grade_R(\fa,M) \geq  r+1.$ Similarly,
  $\grade_R(\fa,N)\geq r+1$.

Let $i=\pd(M)$ and let $j:=\grade_R(\fa,N)$.
 The case $i=0$ is trivial. By Auslander-Buchsbaum, we have 
 \[\begin{array}{ll}
j&\geq \dim R-\grade_R(\fa,M)+r+1 \\
 &\geq\depth(R)-\grade_R(\fa,M)+r+1\\
 &\geq\depth(R)-\depth(M)+r+1\\
 &=\pd(M)+r+1.
 \end{array}\]

 By definition, there is an exact sequence $0\to R^{n_i}\to \ldots \to R^{n_0}\to M\to 0.$ We break down it into short
exact sequences:
\begin{enumerate}
		\item[  ] $0\lo \Syz_1(M)\lo  R^{n_0}\lo M\lo 0$
\item[ ] $\vdots$\item[ ] $0\lo \Syz_{i-1}(M)\lo  R^{n_{i-2}}\lo \Syz_{i-2}(M)\lo 0$ and
\item[ ] $0\lo R^{n_{i}}\lo  R^{n_{i-1}}\lo \Syz_{i-1}(M)\lo 0$.
	\end{enumerate}
This induces:
\begin{enumerate}
\item[] $0\lo \Tor^R_1(M,N)\lo \Syz_1(M)\otimes_R N\lo  R^{n_0}\otimes_R N\lo M\otimes_R N\lo 0$,
\item[] $0\lo \Tor^R_1(\Syz_{1}(M),N) \lo \Syz_2(M)\otimes_R N\lo  R^{n_1}\otimes_R N\lo \Syz_1(M)\otimes_R N\lo 0$
\item[] $\vdots$
\item[] $0\to \Tor^R_1(\Syz_{i-2}(M),N) \to \Syz_{i-1}(M)\otimes_R N\to  R^{n_{i-2}}\otimes_R N\to \Syz_{i-2}(M)\otimes_R N\to 0$ and
\item[] $0\lo \Tor^R_1(\Syz_{i-1}(M),N)\lo R^{n_{i}}\otimes_R N\lo  R^{n_{i-1}}\otimes_R N\lo \Syz_{i-1}(M)\otimes_R N\lo 0$.
	\end{enumerate}
Since one of $M$ and $N$ is locally free over
$\Spec(R)\setminus\V(\fa)$
we deduce that $\Tor^R_1(\Syz_{i-1}(M),N))$ is $\fa$-torsion. Thus, $\HH^{+}_{\fa}(\Tor^R_1(\Syz_{i-1}(M),N))=0$
and   $\HH^{0}_{\fa}(\Tor^R_1(\Syz_{i-1}(M),N))=\Tor^R_1(\Syz_{i-1}(M),N)$.
 Recall that    $\grade_R(\fa,R^{n_{i}}\otimes_R N)>0$ and     $\Tor^R_1(\Syz_{i-1}(M),N) \subset R^{n_{i}}\otimes_R N$.   We use these  to deduce that
 $$\Tor^R_1(\Syz_{i-1}(M),N)= \HH^{0}_{\fa}(\Tor^R_1(\Syz_{i-1}(M),N))\subset \HH^{0}_{\fa}(R^{n_{i}}\otimes_R N)=0,$$
i.e., $\Tor^R_1(\Syz_{i-1}(M),N)=0$.
From this, the sequence $$0\lo R^{n_{i}}\otimes_R N\lo  R^{n_{i-1}}\otimes_R N\lo \Syz_{i-1}(M)\otimes_R N\lo 0 $$ is exact. Let $\ell\leq i+r-1\leq\depth(N)-2$.
This induces the exact sequence $$0=\HH^{\ell}_{\fa}(R^{n_{i-1}}\otimes_R N)\lo\HH^{\ell}_{\fa}(\Syz_{i-1}(M)\otimes_R N)\lo \HH^{\ell+1}_{\fa}(R^{n_{i}}\otimes_R N)=0.$$
 Let us write this observation in the following way
$$0=\HH^0_{\fa}(\Syz_{i-1}(M)\otimes_RN)=\HH^1_{\fa}(\Syz_{i-1}(M)\otimes_RN)=\cdots=\HH^{r-1+i}_{\fa}(\Syz_{i-1}(M)\otimes_RN).$$
We continue this process to get that $\Tor^R_1(\Syz_1(M),N)=0$ and $$0=\HH^0_{\fa}(\Syz_{i-(i-1)}(M)\otimes_RN)=\ldots=\HH^{r-(i-1)+i}_{\fa}(\Syz_{i-(i-1)}(M)\otimes_RN).$$
Let us write this observation in the following way
$$0=\HH^0_{\fa}(\Syz_{1}(M)\otimes_RN)=\ldots=\HH^{r+1}_{\fa}(\Syz_{1}(M)\otimes_RN).$$
Recall that $\Tor^R_1(M,N)$ is  $\fa$-torsion, $\grade_R(\fa,\Syz_{1}(M)\otimes_RN)>0$ and $\Tor^R_1(M,N)\subseteq \Syz_1(M)\otimes_R N$. From this $\Tor^R_1(M,N)=0$.
Hence, the sequence
$$0\lo \Syz_{1}(M)\otimes_R N\lo  R^{n_{0}}\otimes_R N\lo M\otimes_R N\lo 0 $$is exact.
 Let $\ell\leq r $. Then $\ell\leq\grade_R(\fa,N)-1$.
This yields$$0=\HH^{\ell}_{\fa}(R^{n_0}\otimes_R N)\lo\HH^{\ell}_{\fa}(M\otimes_RN)\lo\HH^{\ell+1}_{\fa}(\Syz_{1}\otimes_RN)=0.$$Therefore, $\HH_{\fa}^0(M\otimes_RN)=\ldots=\HH_{\fa}^r(M\otimes_RN)=0$.
\end{proof}

If both modules have finite projective dimension, we have:
\begin{proposition}\label{gc}
Let $(R,\fm,k)$ be a    local ring of positive depth $d$, $M$ and $N$    are of finite projective  dimension. Assume
one of them is locally free over
$\Spec(R)\setminus\V(\fa)$. Let $0\leq r<d$ be such that
$\grade_R(\fa,M)+\grade_R(\fa,N)\geq d+r+1$. Then  $\HH_{\fa}^0(M\otimes_RN)=\ldots=\HH_{\fa}^r(M\otimes_RN)=0$.
\end{proposition}

\begin{proof} We claim that $N$ and $M$ have depth at least $r+1$.
Clearly $N$ and $M$ have depth at least $r$.
First we show that $\grade_R(\fa,M)=\grade_R(\fa,N)=r$ is not the case: suppose on the contradiction that $\grade_R(\fa,M)=\grade_R(\fa,N)=r$.
Thus, $$2r=\grade_R(\fa,M)+\grade_R(\fa,N)\geq d+r+1,$$ i.e., $r\geq d+1$ which is excluded by the assumption.
Hence, one of $M$ and $N$ has a depth at least $r+1$. By symmetry, we assume that $\grade_R(\fa,N)\geq r+1$. Now we show $\grade_R(\fa,M)\geq r+1$.
Suppose on the contrary that $r\leq\grade_R(\fa,M)<r+1$. Therefore,$$r+\grade_R(\fa,N)=\grade_R(\fa,M)+\grade_R(\fa,N)\geq d+r+1.$$From this, $$d\geq\depth(R)-\pd(N)=\depth(N)\geq\grade_R(\fa,N)\geq d+1.$$
This is a contradiction. In sum, $\grade_R(\fa,M)\geq r+1$ and $\grade_R(\fa,N)\geq r+1$. The remaining of the
proof is similar to  Proposition \ref{g}.
\end{proof}

\begin{example}\label{gex}
	The assumption $\pd(M)<\infty$ is essential:
\begin{enumerate}
\item[i)] Let $R$ be any 1-dimensional local domain which is not regular. Then there is an ideal
$I$ which is not principal. Thus, $I^{\otimes 2}$ has a torsion. Let $r:=0$. Then $2\depth(I)=\dim(R)+r+1$. However, $\h^0(I^{\otimes 2})\neq 0$.
\item[ii)] In  view of \cite[Example 1.8]{tensor2} there is a maximal Cohen-Macaulay and locally free module $M$  over $R:=\frac{k[[x,y,z,w]]}{(xy-uv)}$ such that $ M\otimes_RM^\ast\cong \fm$. Let $r:=2$. Then $\depth(M)+\depth(M^\ast)=\dim R+r+1$. However,  $\h^1(M\otimes_RM^\ast)\neq0$.
\end{enumerate}
\end{example}

Let us consider to another situation for which $\h^0(-\otimes_R-)$ vanishes:

\begin{observation}
Let $(R,\fm,k)$ be a $d$-dimensional regular local ring  with $d>2$ and $I$ be a Gorenstein ideal of height two. Then $\h^0(I\otimes_RI)=0$.
\end{observation}

\begin{proof}
Due to a result of Serre, $I$ generated by a regular sequence $x$ and $y$.
 Since $\HH^0_{\fm}(I\otimes_R I)\subset \tor(I\otimes_R I)$, we deduce that $\HH^0_{\fm}(I\otimes_R I)\subset \HH^0_{\fm}(\tor(I\otimes_R I))$.
 The Koszul complex of $R$ with respect to $x$ and $y$ is a free resolution
of $R/I$. Then, $$\tor(I\otimes _RI)=\Tor^R_2(R/I,R/I)\simeq \HH_2(\mathbb{K}(I; R)\otimes _RR/I)=R/I.$$
Recall that depth  of $R/I$ is positive. By the cohomological characterization of depth,
 $\HH^0_{\fm}(R/I)=0$. We put all things together to deduce that $$\HH^0_{\fm}(I\otimes_R I)\simeq \HH^0_{\fm}(\tor(I\otimes _RI))=\HH^0_{\fm}(R/I)=0.$$
So, $\h^0(I\otimes_RI)=0,$ as claimed.
\end{proof}

\section{Torsion in tensor products}

In \cite[Question 8.4]{2010} Vasconcelos posed some questions.
For example, let $R$ be a one-dimensional domain and $M$ a torsion-free module such that $M\otimes_R M$ is torsion-free. Is $M$ free?

\begin{example}(See \cite[4.7]{tensor})
Let  $(R,\fm)$  be a one-dimensional local domain with a canonical module which is not Gorenstein. Then there is a non-free   and torsion-free module
$M$ such that $M\otimes_R M$ is torsion-free.
\end{example}

\begin{remark}In the positive side, we remark that:\begin{enumerate}
		\item[  i)] The above question is true over  hypersurface rings (see \cite[Theorem 3.7]{tensor}).
\item[ ii)] The question is true provided $M$ is an ideal.
	\end{enumerate}
\end{remark}

Also, Vasconcelos asked:

\begin{question}
Let $R$ be a local domain and $M$ be  torsion-free. Is there an integer $e$ guaranteeing that if $M$ is not free,
then the tensor power $M^{\otimes e}$ has nontrivial torsion?
\end{question}

\begin{proposition}\label{}
Let  $(R,\fm)$  be a  $3$-dimensional Cohen-Macaulay  local  ring and $M$ be a reflexive module  such that $\pd(M)<\infty$. If $M^{\otimes 3}$  is torsion-free, then
$M$ is free.
\end{proposition}

\begin{proof}
	Since $M$ is torsion-free it is a submodule of a free module $F$. Let $C:=\frac{F}{M}$. There is nothing to prove if $C=0$.
Without loss of the generality we assume that $C\neq0$. Note that $\pd(M)\leq 1$. Suppose on the contradiction that $\pd(M)\neq 0$, i.e., $\pd(M)= 1$.
We look at the exact sequence $0\to M\to F\to C\to 0\quad(\ast)$. The induced long  exact sequence, presents the natural isomorphisms
$\Tor^R_{i+1}(C,M)\simeq\Tor^R_{i}(M,M)$ for all $i>0$. Since $\pd(M)=1$, $\Tor^R_{\geq 2}(C,M)=0$ and so $\Tor^R_{+}(M,M)=0$.
This vanishing result allow us to compute
 $\pd(M \otimes_R M) $, see Fact \ref{2au}.B). By Auslander-Buchsbaum formula, $$
\depth(M) + \depth(M)
= \depth(R) + \depth(M \otimes_R M). $$
From   $\depth(M)=2$ we see $\depth(M \otimes_R M)=1$.
 Again, $(\ast)$ yields the following exact sequence $$0\lo \Tor^R_1(C,M^{\otimes  2})\lo M^{\otimes  3}\lo M^{\otimes  2}\otimes_R F\lo M^{\otimes  2}\otimes_R C\lo 0$$and 
$\Tor^R_{i+1}(C,M^{\otimes  2})\simeq\Tor^R_{i}(M,M^{\otimes  2})$ for all $i>0$.
Here, we show $\Tor^R_+(-,M^{\otimes  2})$ is of finite length. Indeed,
let $\fp\neq \fm$ be in support of $M$. Since $M_{\fp}$ is reflexive and of finite projective dimension, it is $(\Se_2)$. Since $\depth(R_{\fp})=\dim R_{\fp}<3$
it follows that  $$\pd(M_{\fp})=\depth(R_{\fp})-\depth(M_{\fp})= 0,$$ i.e., $M$ is locally free. From this, $\Tor^R_{+}(-,M^{\otimes  2})$ is of finite length.
Since $\ell(\Tor^R_1(C,M^{\otimes  2}))<\infty$,
$\Tor^R_1(C,M^{\otimes  2})\subset M^{\otimes  3}$ and $M^{\otimes  3}$ is torsion-free,
we get that $\Tor^R_1(C,M^{\otimes  2})=0$.
In order to show $\Tor^R_2(C,M^{\otimes  2})=0$ we use a trick of Peskine-Szpiro. Since the assumptions are not the same, we present   the details.
Recall that  $\ell(\Tor^R_2(C,M^{\otimes  2}))<\infty$.
By $(\ast)$, we have $\pd(C)=2$. Let $0\to F_2\to F_1\to F_0\to C\to 0$ be a free resolution of $C$. Apply $-\otimes_RM^{\otimes  2}$ to it
we have $$\Tor^R_2(C,M^{\otimes  2})=\ker\left(F_2\otimes_R M^{\otimes  2}\to F_1\otimes_R M^{\otimes  2}\right)\subset \bigoplus_{\rank(F_2)} M^{\otimes  2}.$$Note that $M^{\otimes  2}$ is of positive depth.
Any non-zero submodule of a module of positive depth has a same property.  We apply this for the pair $\Tor^R_2(C,M^{\otimes  2})\subset \bigoplus_{\rank(F_2)} M^{\otimes  2}$
to deduce that
 $\Tor^R_{2}(C,M^{\otimes  2})=0$.  Since $\pd(C)=2$, $\Tor^R_{+}(C,M^{\otimes  2})=0$.
This allow us to apply Fact \ref{2au}.B) to see
$$\depth(C) + \depth(M^{\otimes  2})\stackrel{(+)}= \depth(R) + \depth(M^{\otimes  2}\otimes_R C).$$
 By Auslander-Buchsbaum formula,  $\depth(C)=1$. Recall that $\depth(M^{\otimes  2})=1$.
We see the left hand side of $(+)$ is $2$ and the right hand side is at least $3$. This is a contradiction. In sum,
$M$ is free.
\end{proof}

Finiteness of $\pd(M)$ is important: Let $R:=k[[X,Y,Z,W]]/(X^2)$ and $M:=R/ xR $. It is easy to see
that  $M^{\otimes\ell}$ is  reflexive for all $\ell>0$ but $M$ is not free.

\begin{remark}\label{t2}
Let  $(R,\fm)$ be a   local ring of depth $2$ and $M$ be  torsion-free such that $\pd(M)<\infty$. If $M^{\otimes  2}$  is torsion-free, then
$M$ is free.
\end{remark}

\begin{proof}Suppose on the contradiction that $M$   is not free. Since $M$ is torsion-free it is a submodule of a free module $F$. Let $C:=\frac{F}{M}$.
Without loss of the generality we assume that $C\neq0$.
We look at the exact sequence $0\to M\to F\to C\to 0$. The induced long  exact sequence, presents the natural isomorphisms
$\Tor^R_{i+1}(C,M)\simeq\Tor^R_{i}(M,M)$ for all $i>0$.  It follows by Auslander-Buchsbaum that $\pd(M)= 1$. We conclude that $\Tor^R_{\geq 2}(C,M)=0$.
Thus $\Tor^R_{+}(M,M)=0$. We recall from   Fact \ref{2au}.B)
 that
$
\depth(M) + \depth(M)
\stackrel{(+)}= \depth(R) + \depth(M \otimes_R M) $. Also, $\depth(M \otimes_R M)>0$ because it is torsion-free.  The left hand side of $(+)$ is $2$ and the right hand side is at least $3$.
This  contradiction says that $M$ is free.
\end{proof}

Finiteness of $\pd(M)$ is important: Let $R:=k[[X,Y,Z]]/(X^2)$ and $M:=R/ xR $. It is easy to see
that  $M^{\otimes\ell}$ is  reflexive for all $\ell>0$ but $M$ is not free.

\begin{corollary}
Let  $(R,\fm)$ be a 2-dimensional normal hypersurface ring and $M $ be such that
that  $M^{\otimes2}$ is  torsion-free. Then $M$ is   free.
\end{corollary}

\begin{proof}
 In view of  \cite[Proposition 5.2]{Da} we see  $\Tor^R_{+}(M,M)=0$. Due to the  depth formula we have $
2\depth(M)= 2 + \depth(M \otimes_R M) \geq3.$ It turns out that $\depth(M)=2$.
From $\Tor^R_{+}(M,M)=0$ we deduce that  $\pd(M)<\infty$.
By Auslander-Buchsbaum formula, $M$ is   free.
\end{proof}

For a higher dimensional version, see Corollary \ref{4n}.

\section{Higher  cohomology of tensor products}
This section is divided into 4 subsections:

\subsection{The low-dimensional approach}

  \begin{fact} \label{hw}(See \cite[Theorem 2.4]{tensor2})
Let $R$ be such that its completion is a quotient of equicharacteristic regular local
ring by  a nonzero element. Let $r$ be such that $0\leq r< \dim R$. Assume $M\otimes N$
is $(\Se_{r+1})$ over the punctured spectrum and at least one of them is of constant rank and $\pd(M)<\infty$. Then
 $\HH^{r}_{\fm}( N\otimes_R  M)=0$ and both of $M$ and $N$ has depth at least $r$ if and only if $\depth(N)+\depth(M)\geq \dim R+r +1$.
\end{fact}

 \begin{observation} Let  $(R,\fm)$ be a regular local ring of dimension $2$  and $M$ a torsion-free module. Then $\HH^{1}_{\fm}( M\otimes _RM)=0$  for some $0\leq i<\dim R$
 if and only if $M$ is free.
\end{observation}

\begin{proof}The case $i=0$ is in Observation \ref{2au}. The case $i=1$ is in the above fact.
\end{proof}

It may be natural to extend the above result to 3-dimensional case by replacing torsion-free with the reflexive modules. This is not the case:

 \begin{corollary}
 Let  $(R,\fm)$ be a regular local ring of dimension $3$  and $M$ a reflexive module.
 \begin{enumerate}
\item[i)] Always $\HH^{0}_{\fm}( M\otimes _RM)=0$.
\item[ii)] If $\HH^{i}_{\fm}( M\otimes _RM)=0$ for some $0<i<3$, then $M$ is free.
\end{enumerate}
\end{corollary}

\begin{proof}
The first item is in Proposition \ref{g}. We may assume that $i>0$ and that $M\neq 0$.  Reflexive modules over 2-dimensional regular
local rings are free. From this, $M$ is  locally free over the punctured spectrum.
We apply Fact \ref{hw} for $r=i$, to see that $2\depth(M)\geq \dim R+i+1\geq 5$. That is $2<\frac{5}{2}\leq\depth(M)\leq \dim(M)\leq 3$.
Thus, $\depth(M)= 3$.  Due to Auslander-Buchsbaum, $M$ is free.
\end{proof}

Without any  restriction on the dimension,  assume $\HH^{1}_{\fm}( M\otimes _RM)=0$. What can say about freeness of $M$? We will answer this in Example \ref{41}.
In  view of \cite[Example 1.8]{tensor2} there is a non-free ideal $I$  of $R:=\frac{k[[x,y,z,w]]}{(xy-uv)}$ such that $I\otimes I^\ast$ is torsion-free.

\begin{example}
Let $(R,\fm,k)$ be a local  ring  of depth at least $3$.
 Then  i) $\fm\otimes_R \fm^\ast$ is torsion-free, ii) $\fm$ is locally free and non-free, and  iii)
 $\HH^{2}_{\fm}(\fm\otimes_R \fm^\ast)=0$.
\end{example}

\begin{proof} Clearly $\fm$ is non-free and locally free, and that
$\Ext^{<3}_R(k,R)=\HH^{<3}_{\fm}(R)=0$.
We look at $0\to \fm \to R\to k\to 0 \quad(\ast)$. It yields that $0=k^\ast\to \fm^\ast\to R^\ast\to\Ext^1_R(k,R)=0$, i.e.,
$\fm^\ast\simeq R$. Also, $(\ast)$ implies that
$0=\HH^1_{\fm}(k)\to\HH^2_{\fm}(\fm)\to \HH^2_{\fm}(R)=0$. So,  $\HH^{2}_{\fm}(\fm\otimes_R \fm^\ast)\simeq\HH^2_{\fm}(\fm)=0$.
\end{proof}

\subsection{The regular case}
 Let $M$ be a surjective Buchsbaum $A$-module of finite projective dimension
	and $N$ a maximal surjective Buchsbaum $A$-module. Yoshida asked in \cite[Question 3.5]{yosh2} when is $M \otimes N$
	surjective Buchsbaum? He presented a negative answer to this by using a beautiful criterion of Yamagishi and some results of Strooker, Goto and Kawasaki, see \cite[ 3.7]{yosh2}.
	In his example, $\depth(M)+\depth(N) \geq \dim  A$. Thus, the particular case of Proposition \ref{main5b}(i) yields a new negative answer to \cite[Question 3.5]{yosh2}. 

\begin{proposition}\label{main5b}
Let $(R,\fm,k)$ be a regular  local ring and $M$ be an indecomposable Buchsbaum module of dimension $d$ which is not
Cohen-Macaulay.
\begin{enumerate}
\item[i)] If $\depth(M)=1$, then \begin{equation*}
\h^i(M\otimes _RM)=\left\{
\begin{array}{rl}
{d}\choose{2}& \  \   \   \   \   \ \  \   \   \   \   \ \text{if }\   \   i=0\\
d+1 & \  \   \   \   \   \ \  \   \   \   \   \ \text{if } \   \ i=1\\
0& \  \   \   \   \   \ \  \   \   \   \   \ \text{if }\   \  2\leq i<  d
\end{array} \right.
\end{equation*}In particular, $M\otimes_R M$ is  not Buchsbaum.
\item[ii)] If $d>3$ and $M$ is almost Cohen-Macaulay, then
\begin{equation*}
\h^i(M\otimes_R M^\ast)=
\left\{
\begin{array}{rl}
0&    \   \   \ \ \text{ if } \   \ i\in\{0\}\cup[3,d-2]\\
1&    \   \   \    \ \text{ if } \   \ i=1\\
d&     \   \   \    \ \text{ if }\   \ i=2\ \ or \ \ i=d-1
\end{array}\right.
\end{equation*} 
and $M\otimes_R M^\ast$ is quasi-Buchsbaum. In particular, and against to
$M$ and $M^\ast$, $M\otimes_R M^\ast$ is not Buchsbaum.
\end{enumerate}
\end{proposition}

\begin{proof} i) First, we state a more general claim:
\begin{enumerate}
\item[Claim A)] Let $(A,\fn,k)$ be a   Cohen-Macaulay local ring of dimension at least two and $I\lhd A$ be $\fn$-primary.
Then\begin{equation*}
\h^i(I\otimes_A \fn)=\left\{
\begin{array}{rl}
\beta_2(A/I)& \  \   \   \   \   \ \  \   \   \   \   \ \text{if } \   \  i=0\\
\mu(I)+\ell(A/I) & \  \   \   \   \   \ \  \   \   \   \   \ \text{if }\   \  i=1\\
0& \  \   \   \   \   \ \  \   \   \   \   \ \text{if } \   \ 2\leq i<  \dim A\\
\end{array} \right.
\end{equation*}\end{enumerate}Indeed,
let $d:=\dim A$.
 We look at $0\to \fn \to A\to k\to 0$
 and we drive the following exact sequence$$0\lo \Tor_1^A(k,I)\lo I\otimes_A \fn\lo I\lo I\otimes_A k\lo 0\quad(\ast)$$
Recall that $I\otimes_A k\simeq\frac{I}{I\fn}\simeq k^{\mu(I)}$ and $\Tor_1^A(k,I)\simeq\Tor_2^A(k,A/I)\simeq k^{\beta_2(A/I)}$.
We break down
$(\ast)$ into
 a) $0\to k^{\beta_2(A/I)}\to I\otimes_A \fn\to L\to 0$ and
b) $0\to L\to  I\to  k^{\mu(I)}\to 0$.
We conclude from a) the exact sequence
  $0\to \HH^0_{\fn}(k^{\beta_2(A/I)})\to \HH^0_{\fn}(I\otimes_A \fn)\to \HH^0_{\fn}(L)$.
It follows from b) that the sequence $0\to \HH^0_{\fn}(L)\to \HH^0_{\fn}(I)=0$ is exact. We combine these  to see $$\ell(\HH^0_{\fn}(I\otimes_R \fn))=\ell(\HH^0_{\fn}(k^{\beta_2(A/I)}))={\beta_2(A/I)}.$$
From  a) we have $\HH^1_{\fn}(I\otimes_R \fn)\simeq \HH^1_{\fn}(L)$.
 From b),  $$0=\HH^0_{\fn}(I)\lo \HH^0_{\fn}(k^{\mu(I)})\lo \HH^1_{\fn}(L)\simeq\HH^1_{\fn}(I\otimes_R \fn)\lo \HH^1_{\fn}(I)\lo \HH^1_{\fn}(k^{\mu(I)})=0.$$
 In order to compute $\HH^1_{\fn}(I)$, we look at $0\to I\to A\to A/I\to 0$. This induces
 $0=\HH^0_{\fn}(A)\to\HH^0_{\fn}(A/I)\to \HH^1_{\fn}(I) \to \HH^1_{\fn}(A)= 0 .$
 Thus, $\HH^1_{\fn}(I)\simeq\HH^0_{\fn}(A/I)=A/I$.
We put all of these together to see $$0\lo k^{\mu(I)}\lo \HH^1_{\fn}(I\otimes_A \fn)\lo A/I\lo0.$$
We conclude that $\h^1(I\otimes_A \fn)=\mu(I)+\ell(A/I)$.
Let $2\leq i<  d$.
Recall that $$\HH^i_{\fn}(I\otimes_A \fn)\simeq\HH^i_{\fn}(L)\simeq\HH^i_{\fn}(I).$$
 We look at $0=\HH^{i-1}_{\fn}(A/I)\to \HH^i_{\fn}(I) \to \HH^i_{\fn}(A)= 0 $ to deduce that  $\HH^i_{\fn}(I\otimes_A \fn)\simeq\HH^i_{\fn}(I)=0$.
This completes the proof of Claim A).
Recall from \cite[Corollary (3.7)]{goto} that:
\begin{enumerate}
\item[Fact A)] Let $(A,\fn)$ be a regular local ring and $P$ be an indecomposable Buchsbaum module of maximal
dimension. Then $P\simeq \Syz_{i}(\frac{A}{\fn})$ where $i=\depth(P)$.
\end{enumerate}
In the light of Fact A) we see $M=\Syz_{1}(k)=\fm$. Note that $\beta_2(k)$ is equal to ${d}\choose{2}$
and $\mu(\fm)=d$. It follows by the assumptions that
$\dim(R)\geq2$. Claim A) yields that: \begin{equation*}
\h^i(M\otimes _RM)=\left\{
\begin{array}{rl}
{d}\choose{2}& \  \   \   \   \   \ \  \   \   \   \   \ \text{if }\   \   i=0\\
d+1 & \  \   \   \   \   \ \  \   \   \   \   \ \text{if } \   \ i=1\\
0& \  \   \   \   \   \ \  \   \   \   \   \ \text{if }\   \  2\leq i<  d
\end{array} \right.
\end{equation*}
 To see the particular case,
we recall from \cite[Theorem (1.1)]{goto} that:
\begin{enumerate}
\item[Fact B) ] Let $(A,\fn)$ be a regular local ring and $P$ be  Buchsbaum. Then $P\simeq\bigoplus_{0\leq i\leq \dim (A)}\Syz_i(\frac{A}{\fn})^{\h^i}$ where $\h^i:=\h^i(P)$ for all $0\leq i< \dim A$.
\end{enumerate}Suppose on the contradiction that $M\otimes_R M$ is   Buchsbaum.
Due to Fact B),
$M\otimes_R M\simeq\bigoplus_{0\leq i\leq d}\Syz_i(k)^{\h^i}$ where $\h^i:=\h^i(M\otimes_R M)$ for $i\neq d$.  It turns out that  $$M\otimes_R M\stackrel{(\natural)}\simeq k^{{d}\choose{2}}\bigoplus\Syz_{1}(k)^{\oplus (d+1)}\bigoplus R^{n}$$ for some $n\geq0$.
Since $M\simeq\fm$, we see the rank of left hand side of $(\natural)$ is one. The rank of
right hand side is $0+(d+1)+n$.  Since $n\geq 0$, we get to a contradiction. So, $M\otimes_R M$ is  not Buchsbaum.

ii) We recall that $M$ is called almost Cohen-Macaulay if $\depth(M)\geq\dim( M)-1$.
Since $M$ is not Cohen-Macaulay,  $\depth(M)=\dim (M)-1=d-1$.
In the light of Fact A),  $M=\Syz_{d-1}(k)$.  Since $M$ is locally free, $\Tor_1^R(M,M^\ast)$ is of finite length.
 We look at $0\to R \to R^d\to M\to 0$
 and we drive the following exact sequence$$0\lo \Tor_1^R(M,M^\ast)\lo M^\ast  \lo (M^\ast)^d\lo M\otimes _RM^\ast\lo 0.$$
We break down
it into
  $0\to \Tor_1^R(M,M^\ast)\to M^\ast\to L\to 0$  and
  $0\to L\to  (M^\ast)^d\to M\otimes_R M^\ast\to 0$.
It follows from the first  sequence that
  $$0=\HH^1_{\fm}(\Tor_1^R(M,M^\ast))\to \HH^1_{\fm}(M^\ast)\to \HH^1_{\fm}(L)\to\HH^2_{\fm}(\Tor_1^R(M,M^\ast))=0.$$Similarly, $\HH^{+}_{\fm}(M^\ast)\simeq\HH^{+}_{\fm}(L)$. Recall that $M^\ast$
  is reflexive. In particular it is $(\Se_2)$. So, $ \HH^{1}_{\fm}(L)\simeq\HH^{1}_{\fm}(M^\ast)=0$.
It follows from the second short exact sequence that $$0=\HH^0_{\fm}((M^\ast)^d)\to \HH^0_{\fm}(M\otimes_R M^\ast)\to \HH^{1}_{\fm}(L)=0.$$ From this, $\h^0(M\otimes _RM^\ast)=0$.
\begin{enumerate}
\item[Fact C)]  (See \cite[Proposition A.1]{ag}) Let $A$ be a ring, a necessarily and sufficient condition for which $P$ be projective
is that   $\varphi_P:P\otimes_AP^\ast\to \Hom_A(P,P)$ is (surjective)  isomorphism.
\end{enumerate}

 Since $M$ is locally free, it follows from Fact C) that  $K:=\ker(\varphi_{M})$ and
$C:=\coker(\varphi_{M})$ are of finite length and that $C\neq 0$. From this,  $\HH^0_{\fm}(C)=C\neq 0$, $\HH^+_{\fm}(C)= \HH^+_{\fm}(K)= 0$.
 We look at   $0\to K\to M\otimes_R  M^\ast\to\im(\varphi_{M})\to 0 $ and
$0\to\im(\varphi_{M})\to \Hom_R(M,M) \to C\to 0 $. Since $\depth(M)>1$ another result of Auslander-Goldman (\cite[Proposition 4.7]{ag})
says that $\depth(\Hom_R(M,M) )>1$, i.e., $\HH^0_{\fm}(\Hom_R(M,M))=\HH^1_{\fm}(\Hom_R(M,M))=0$.
We apply this along with the long exact sequences of local cohomology modules to see
\begin{enumerate}
\item[ ] $0=\HH^1_{\fm}(K)\to \HH^1_{\fm}(M\otimes _RM^\ast)\to\HH^1_{\fm}(\im(\varphi_M))\to \HH^2_{\fm}(K)=0$
\item[ ] $0= \HH^0_{\fm}(\Hom_R(M,M))\lo\HH^0_{\fm}(C)\lo\HH^1_{\fm}(\im(\varphi_M))\lo \HH^1_{\fm}(\Hom_R(M,M))=0,$
\end{enumerate}e.g.,   $$\HH^1_{\fm}(M\otimes _RM^\ast)\simeq\HH^1_{\fm}(\im(\varphi_M))\simeq\HH^0_{\fm}(C) \simeq C\simeq \Tor^R_1(D(M),M),$$ because
$\coker(\varphi _M)=\Tor^R_1 (D(M),M).$
Let $\fm=(x_1,\ldots, x_d)$.  In view of $0\to R \stackrel{(x_1,\ldots, x_d)}\lo R^d\to M\to 0$ we see $$D(M)= \coker\left(R^d\stackrel{(x_1,\ldots, x_d)}\lo R\right)=\frac{R}{\fm}.$$
Also, $$\Tor^R_1(D(M),M)\simeq \Tor^R_1(k,\Syz_{d-1}(k))=\Tor_d^R(k,k)=k.$$
Combining these, $\h^1(M\otimes_R M^\ast)=\ell(\Tor^R_1(D(M),M))=1$. Also, $\fm\HH^1_{\fm}(M\otimes _RM^\ast)=0$.
\begin{enumerate}
\item[Fact D)] (See \cite[Proposition 4.1]{bv}) Let $(A,\fn)$ be a local ring, $L$ be
 locally free and $N$ be of depth at least $3$. Then
$\Ext^i_A (L, N)\simeq \HH^{i+1}_{\fm}( N\otimes_A  L^{\ast})$ for all $1\leq i\leq \depth(N)-2$.
\end{enumerate} By this
 $\HH^2_{\fm}(M\otimes_R M^\ast)\simeq\Ext^1_R(M,M)$, because $\depth(M)=d-1\geq3$.
Apply $\Hom_R(-,M)$ to   $0\to R \to R^d\to M\to 0$ to see $$0\to \Hom_R(M,M)  \to \Hom_R(R^d,M)\to \Hom_R(R,M)\to \Ext^1_R(M,M)\to 0.$$
Thus,
$$\HH^2_{\fm}(M\otimes_R M^\ast)\simeq\Ext^1_R(M,M)=\coker\left(M^d\stackrel{(x_1,\ldots, x_d)}\lo M\right)=\frac{M}{\fm M}.$$ Hence,
$$\h^2(M\otimes_R M^\ast)=\ell(\frac{M}{\fm M})=\mu(M)= \beta_{d-1}(k)=d.$$ Also, $\fm\HH^2_{\fm}(M\otimes _RM^\ast)=0$.

Let $3\leq i\leq  d-2$.  Due to  Fact D) we know that
 $\HH^{i}_{\fm}(M\otimes_R M^\ast)\simeq \Ext^{i-1}_R(M,M)=0$, because $\pd(M)=1$.
Thus,
$\h^i(M\otimes_R M^\ast)=0$.

Here, we compute $\h^{d-1}(M\otimes_R M^\ast)$. To this end,
we recall from  \cite[Proposition 4.1]{tensor2} that:
\begin{enumerate}
\item[Fact E)] Let  $A$ and $B$ be
 locally free over  a regular local ring $(S,\fn)$ of dimension $d\geq 3$ and let $2\leq j\leq d-1$. Then
 $\HH^j_{\fn}(A\otimes_S B)^v\simeq \HH^{d+1-j}_{\fn}(A^{\ast}\otimes_S B^{\ast})$, where  $(-)^v$ is the Matlis duality.
\end{enumerate} Since  $d-1\geq2$, $\Syz_{d-1}(k)$  is a second syzygy, it is reflexive. Also, $\ell((-)^v)=\ell(-)$. We use these  to see $$\h^{d-1}(M\otimes_R M^\ast)=
\ell(\HH^{d-1}_{\fm}(M\otimes_R M^\ast)^v)=\ell(\HH^{2}_{\fm}(M^\ast\otimes_R M^{\ast\ast}))=\ell(\HH^{2}_{\fm}(M^\ast\otimes_R M))=d.$$Since
Matlis duality
preserves
 the annihilator we deduce that $\fm\HH^{d-1}_{\fm}(M^\ast\otimes_R M)=0$.

We proved  that $\fm\HH^{<d}_{\fm}(M\otimes_R M^\ast)=0$. By definition, $M\otimes_R M^\ast$ is quasi-Buchsbaum.
In view of $0\to R \to R^d\to M\to 0$ we see
$0\to M^\ast\to R^d\to R$ is exact. Thus, $M^\ast=\Syz_2(R/ \fm)$ which is Buchsbaum. Note that $\rank(M)=\rank(M^\ast)=d-1$, because $0\to M^\ast\to R^d\to \fm \to 0$.
Thus, $\rank(M\otimes_R M^\ast)=(d-1)^2.$
Also, $\rank(\Syz_1(k))=1$, because $\Syz_1(k)=\fm$.
Suppose on the contradiction that $M\otimes_R M^\ast$ is Buchsbaum.
Due to Fact B)  there is an $n\geq 0$ such that $$M\otimes_R M^\ast=\Syz_1(k)\bigoplus\Syz_{2}(k)^{\oplus d}\bigoplus \Syz_{d-1}(k)^{\oplus d}\bigoplus R^{n}.$$ The left hand side is a vector bundle
of rank $(d-1)^2$. The right hand side is a vector bundle
of rank $1+d (d-1)+d (d-1)+n$. Since $n\geq 0$, we get to a contradiction. Thus,
$M\otimes_R M^\ast$ is not  Buchsbaum.
\end{proof}
The above result has a role in \cite{acs}.
Over a regular local ring $(R,\fm)$ of dimension $d>1$, Auslander  was looking for a vector bundle $M$ without free summand of dimension $d$
 such that $\pd(M)=\pd(M^\ast)$ and $\HH^{0}_{\fm}( M\otimes _RM^\ast)=0$.
He proved  the existence of $M$ is equivalent to the oddness of $d$.

 \begin{corollary} \label{ausbun}Let  $(R,\fm,k)$ be a regular local ring of odd dimension $d$  and $M$ be as above. If $M$ is Buchsbaum,  then $M\simeq \Syz_{\frac{d+1}{2}}(k)^{\oplus m} $
 for some $m$.
\end{corollary}

\begin{proof} Suppose first that $M$ is indecomposable.
By Fact \ref{main5b}.A)
$M\simeq\Syz_i(k)$  where $i:=\depth(M)$. Since $M$ has no free direct summand,
$i<d$. This allow us to use \cite[Lemma 3.2]{goto} to see $M^\ast=\Syz_{d-i+1}(k)$.
We deduce from $$d-i=\pd(M)=\pd(M^\ast)=\pd(\Syz_{d-i+1}(k))=d-(d-i+1)$$ that  $i=\frac{d+1}{2}$. In particular,
$M= \Syz_{\frac{d+1}{2}}(k)$.
Now, suppose that $M$ is  decomposable and has a direct summand other than $\Syz_{\frac{d+1}{2}}(k)$.
In view of  Fact \ref{main5b}.B)  there is an $I\subset[1,d-1]$ such that  $M\simeq\bigoplus_{i\in I}\Syz_i(k)^{\h^i}$.
Note that $$\pd(M)=\sup_{i\in I}\{\pd(\Syz_i(k))\}=\sup_{i\in I}\{d-i\}=d-\inf\{i:i\in I\}.$$ Let $j$ be such that $j=d-\inf\{i:i\in I\}$.
Recall that $\Syz_i(k)^\ast=\Syz_{d-i+1}(k)$.
Since $\pd(M)=\pd(M^\ast)$ it follows that $\Syz_{d-j+1}(k)$ is a direct summand of $M$.  One of $j$ and $d-j$ is
smaller than $\frac{d+1}{2}$. Without loss of the generality, we assume that $j<\frac{d+1}{2}$  (one may use \cite[Theorem 2.4]{tensor2} to get  a contradiction. Here, we follow
our simple reasoning:) We look at $0\to\Syz_j(k)\to R^{\beta_{j-1}(k)}\to\Syz_{j-1}(k)\to 0$.
This induces $$0\to\Tor^R_1(\Syz_j(k),\Syz_{j-1}(k))\to\Syz_j(k)\otimes_R \Syz_{j}(k)\to R^{\beta_{j-1}(k)}\otimes_R \Syz_{j}(k)\to \Syz_{j}(k)\otimes_R \Syz_{j-1}(k)\to 0.$$
Note that $$\Tor^R_1(\Syz_j(k),\Syz_{j-1}(k))\simeq\Tor^R_j(\Syz_j(k), k )\simeq\Tor^R_{j+j}(k , k)\simeq  k ^{\oplus\beta_{2j}(k)}.$$
Since $j<\frac{d+1}{2}$ we  conclude  that $\Tor^R_1(\Syz_j(k),\Syz_{j-1}(k))$ is nonzero and of finite length.  Since $$k\subset\Tor^R_1(\Syz_j(k),\Syz_{j-1}(k))\subset \Syz_j(k)\otimes_R \Syz_{j}(k)\subset M\otimes_R M^\ast,$$ we see that $\HH^0_{\fm}(M\otimes _RM^\ast)\neq 0$, a contradiction.
\end{proof}

 \begin{corollary} Let  $(R,\fm,k)$ be a regular local ring of dimension $d>2$  and let $i$ be such that the vanishing of $\HH^{i}_{\fm}(M^\ast\otimes_R M)$ implies freeness of each locally free and torsion-free module  $M$. Then $i=1,2,$ or $d-1$.
\end{corollary}

\subsection{The singular case}

Recall that vanishing of $\HH^{2}_{\fm}(M\otimes_R  M^{\ast})$ over regular local rings  implies freeness of
$M^\ast$. This can't be extended into hypersurface rings:
Let
$R:=\frac{k[[x,y,z,w]]}{(xy-uv)}$ and $I:=(x,u)$. Then $\HH^{2}_{\fm}(I\otimes_R I^\ast)=0$ but $I^\ast$ is not free. In the forthcoming work \cite[Theorem 8.1]{acs},  there is an essential generalization of the next result.

\begin{remark}\label{implicit} 
Let $R$ be a hypersurface of dimension $d\geq2$ and $M$ be torsion-free, locally free and of constant rank.
Assume  $\HH^1_{\fm}(M\otimes _R M^\ast)= \HH^2_{\fm}(M\otimes _R M^\ast)=0$.
Then $M^\ast$ is free.
\end{remark}
\begin{proof}
This  stated  implicitly in \cite{tensor2} and we left the routine modification to the reader. 
\end{proof}

\begin{observation} \label{kan} Let  $(R,\fm)$ be a  Cohen-Macaulay local ring of dimension $d>1$
with isolated Gorenstein singularity and possessing a canonical module.   Then $\HH^{i}_{\fm}(\omega_R\otimes _R \omega_R^\ast)\neq0$ if and only if $i\leq1$ or $i=d$.
\end{observation}

\begin{proof}
By isolated Gorenstein singularity we mean a non Gorenstein ring which is Gorenstein over the punctured spectrum.
Since $d>1$ it follows that $R$ is quasi-normal. It turns out that  $\omega_R$ is reflexive. Also, $\omega_R$ may regard as an ideal
of height one. We look at $0\to\omega_R\to R \to \frac{R}{\omega_R}\to 0$. This induces $$0\lo(\frac{R}{\omega_R})^\ast\lo R^\ast\lo \omega_R^\ast  \lo \Ext^1_R(\frac{R}{\omega_R},R)\lo 0.$$
Set $E:=\Ext^1_R(\frac{R}{\omega_R},R)$ and note that $(\frac{R}{\omega_R})^\ast=\{r\in R:r\omega_R=0\}=0$. Recall that $E$ is of finite length.
It follows that $\Tor^R_{\leq1}(E,\omega_R) $ is of finite length. Suppose on the contradiction that $E=0$.
This implies that $R^\ast\simeq  \omega_R^\ast$. Thus, $\omega_R\simeq \omega_R^{\ast\ast}\simeq R^{\ast\ast}\simeq  R$. Since $R$ is not Gorenstein, we get to a contradiction.
Hence $E\neq 0$. Also, we have $$0\lo\Tor^R_1( \omega_R^\ast,\omega_R)\lo\Tor^R_1(E,\omega_R)\stackrel{f}\lo \omega_R\stackrel{g}\lo \omega_R\otimes _R \omega_R^\ast\lo E\otimes _R \omega_R \lo 0. $$
Since $\ell(\Tor^R_1(E,\omega_R))<\infty$, we have $\ell(\im(f))<\infty$. We deduce from $\im(f)\subset  \omega_R$ and  $\depth(\omega_R)>0$ that $\ker(g)=\im(f)=0$. Therefore,
$0\to\omega_R\to \omega_R\otimes _R \omega_R^\ast\to E\otimes _R \omega_R \to 0 $ is exact.
We apply the  long exact sequence of local cohomology modules:
$$0=\HH^0_{\fm}(\omega_R)\lo \HH^0_{\fm}(\omega_R\otimes _R \omega_R^\ast)\lo \HH^0_{\fm}(E\otimes _R \omega_R)\lo\HH^1_{\fm}(\omega_R)=0.$$
Since $E\neq 0$,  $E\otimes\omega_R\neq 0$ and it is of finite length. We put these  together to see that $$\HH^0_{\fm}(\omega_R\otimes _R \omega_R^\ast)\simeq \HH^0_{\fm}(E\otimes _R \omega_R)=E\otimes _R \omega_R\neq0.$$

 Since  $(\omega_R)_{\fp}\simeq \omega_{R_{\fp}}\neq 0$ we deduce that $\Supp(\omega_R)=\Spec(R)$.
Also, $$\Ass(\Hom_R(\omega_R,R))=\Supp(\omega_R)\cap\Ass(R)=\Spec(R)\cap\Ass(R)=\Ass(R).$$ From this, $\Supp(\omega_R^\ast)=\Spec(R)$. It follows that $\Supp(\omega_R\otimes \omega_R^\ast)=\Spec(R)$. Thus,  $\dim(\omega_R\otimes_R  \omega_R^\ast)=d$.
By  Gorthendieck's non-vanishing theorem, $\HH^{d}_{\fm}(\omega_R\otimes_R  \omega_R^\ast)\neq0$.

Let  $\varphi_{\omega_R}:\omega_R\otimes_R\omega_R^\ast\to \Hom_R(\omega_R,\omega_R)$.
Recall that $\Hom_R (\omega_R, \omega_R)\simeq R$  and that
 $\HH^0_{\fm}(R)=\HH^1_{\fm}(R)=0$. Since $\omega_R$ is locally free, it follows from Fact \ref{main5b}.C) that  $K:=\ker(\varphi_{\omega_R})$ and
$C:=\coker(\varphi_{\omega_R})$ are of finite length and that $C\neq 0$. From this,  $\HH^0_{\fm}(C)=C\neq 0$, $\HH^+_{\fm}(C)= \HH^+_{\fm}(K)= 0$.
 We look at   $0\to K\to \omega_R\otimes_R  \omega_R^\ast\to\im(\varphi_{\omega_R})\to 0 $ and
$0\to\im(\varphi_{\omega_R})\to R \to C\to 0 $. It follows that $$\HH^1_{\fm}(\omega_R\otimes _R \omega_R^\ast)\simeq\HH^1_{\fm}(\im(\varphi_{\omega_R}))\simeq\HH^0_{\fm}(C)\simeq C\neq 0.$$

Note that there is nothing to prove if  $d=2$.
Assume that
$d>2$ and let $2\leq i\leq d-1$.  Then
$\HH^{i}_{\fm}(\omega_R\otimes _R \omega_R^\ast)\simeq \HH^i_{\fm}(\im(\varphi_{\omega_R}))\simeq\HH^{i-1}_{\fm}(C)=0.$
The proof is now  complete.\end{proof}

In the following corollaries there is no   trace of local cohomology:
\begin{corollary}\label{tor1}
Assume in addition to Observation \ref{kan} that type of $R$ is two. Then $\Tor^R_1(\omega_R,\omega_R)\neq0$.
\end{corollary}

\begin{proof} Suppose on the contradiction that $\Tor^R_1(\omega_R,\omega_R)=0$.
	Since type of $R$ is two, it follows that $\mu(\omega_R)=2$. In particular,
	there is an exact sequence $0\to\omega_R ^\ast\to R^2\to\omega_R\to 0$ (see \cite[Lemma 3.3]{hh}).
	This induces $$0=\Tor^R_1(\omega_R,\omega_R)\to\omega_R\otimes_R  \omega_R^\ast\to\omega_R^{\oplus2}\to \omega_R^{\otimes2} \to0. $$
Then $\omega_R\otimes_R  \omega_R^\ast\subset \omega_R^{\oplus2}$ is torsion-free. This is in the contradiction with Observation \ref{kan}.
\end{proof}

\begin{corollary}\label{quasi}Let  $(R,\fm)$ be a  generically Gorenstein  Cohen-Macaulay local ring  possessing a canonical module.  Suppose $R$  is  of type at most two. Then $\Ext^1_R(\omega_R,R) =0$
	if and only if $R$ is Gorenstein.
\end{corollary}

\begin{proof} If $R$ is Gorenstein, then $\omega_R=R$ and so $\Ext^1_R(\omega_R,R) =0$.
	Conversely, assume that  $\Ext^1_R(\omega_R,R)=0$. By induction on $d:=\dim R$ we argue that $R$ is Gorenstien. In view of \cite[Corollary 2.2]{hh} we may assume that $d>1$. Suppose, inductively, $R_{\fp}$ is Gorenstein for all $\fp\in\Spec(R)\setminus\{\fm\}$.
	In particular, $\omega_R$ is locally free over the punctured spectrum. Suppose on the
	contradiction that $R$ is not Gorenstein. By definition,   $R$ is of isolated Gorenstein singularity.
	It follows from $\Ext^1_R(\omega_R,R) =0$  that $\Tor^R_1(\omega_R,\omega_R)=0$ (see e.g. the proof of \cite[6.1]{hsv}).
	Vanishing of $\Tor^R_1(\omega_R,\omega_R)=0$  excluded by Corollary \ref{tor1}. This
	contradiction shows that  $R$ is Gorenstein.
\end{proof}

\begin{conjecture}(Part of  \cite[Conjecture 3.4]{yosh})
Let $R$ be  a Cohen-Macaulay local ring,  $M$ be perfect and $N$ be Buchsbaum and of maximal dimension. If $\pd(M)\leq\depth
(N)$, then $\h^i(M\otimes_RN)=
\sum_{j=0}^{\pd(M)}\beta_j(M)\h^{j+i}(N)$ for all $i< \dim(M).$
\end{conjecture}

\begin{proposition}\label{vector2}
Let  $(R,\fm)$ be a  Cohen-Macaulay local ring, $M$ be perfect and $N$ be locally free and of constant rank. Then
$\h^i(M\otimes_RN)\leq
\sum_{j=0}^{\pd(M)}\beta_j(M)\h^{j+i}(N)$ for all $i< \dim(M).$
\end{proposition}

\begin{proof} For every module $L$ of finite projective dimension, we have
$\grade (L) + \dim (L) = \dim( R)$. In particular, if $L$ is perfect then $\dim (L)=\dim( R)-\pd(L)$.
Therefore, things are reduced to show $$\h^i(M\otimes_RN)\leq
\sum_{j=0}^{\pd(M)}\beta_j(M)\h^{j+i}(N)$$ for all $i< \dim (R)-\pd (M).$
We may assume that  $\pd( M)>0$. There is nothing to prove if $\dim (R)-\pd( M)=0$. Without loss
of the generality, $\pd (M)<\dim( R)=\depth(R)$. Now, the case $i=0$ is in Proposition \ref{vector}.  We may assume that $i>0$.
 Let $$f:\Syz_1(M)\otimes_RN\lo R^{\beta_{0}(M)}\otimes_RN$$ be the natural map.  Recall from Proposition \ref{vector}   that $\HH^i_{\fm}(\Syz_1(M)\otimes_RN)\simeq\HH^i_{\fm}(\ker(f))$ and there is an exact sequence $$ \HH^i_{\fm}(R^{\beta_{0}(M)}\otimes_RN)\to \HH^i_{\fm}(M\otimes_RN)\to \HH^{i+1}_{\fm}(\ker(f)).$$ Hence
$$
\h^i(M\otimes_RN)\leq \ell(\HH^{i+1}_{\fm}(\ker(f)))+\beta_{0}(M)\h^i(N)=\ell(\HH^{i+1}_{\fm}(\Syz_1(M)\otimes_RN))+\beta_{0}(M)\h^i(N).
$$
In the same vein, $\ell(\HH^{i+1}_{\fm}(\Syz_1(M)\otimes_RN))\leq\ell(\HH^{i+2}_{\fm}(\Syz_1(M)\otimes_RN))+\beta_{1}(M)\h^{i+1}(N).$
Therefore,
\[\begin{array}{ll}
\h^i(M\otimes_RN)&\leq \ell(\HH^{i+1}_{\fm}(\Syz_1(M)\otimes_RN))+\beta_{0}(M)\h^i(N)\\
&\leq\ell(\HH^{i+2}_{\fm}(\Syz_2(M)\otimes_RN))+\beta_{1}(M)\h^{i+1}(N)+\beta_{0}(M)\h^i(N).
\end{array}\]  Repeating this, $
\h^i(M\otimes_RN)
\leq\ell(\HH^{i+\ell}_{\fm}(\Syz_{\ell}(M)\otimes_RN))+\sum_{j=0}^{\ell-1}\beta_j(M)\h^{j+i}(N).$
We put  $\ell:=\pd (M)-i$ to see $$
\h^i(M\otimes _R N)
\leq\ell(\HH^{\pd (M)}_{\fm}(\Syz_{\pd( M)}(M)\otimes _RN))+\sum_{j=0}^{\ell-1}\beta_j(M)\h^{j+i}(N)=\sum_{j=0}^{\pd( M)}\beta_j(M)\h^{j+i}(N),$$as claimed.
\end{proof}

\begin{remark}The same proof shows that: Let $R$ be equi-dimensional and  generalized Cohen-Macaulay local ring and $N$ be locally free and of constant rank. If $\pd(M)<\depth(R)$, then
$$\h^i(M\otimes_RN)\leq
\sum_{j=0}^{\pd(M)}\beta_j(M)\h^{j+i}(N)$$ for all $i< \depth(R)-\pd(M).$\end{remark}

Having Fact \ref{hw} in mind, it may be nice to determine the case
for which $\depth(M)+\depth(N)$ is minimum. Recall that $M$ is called $p$-\textit{spherical} if $\pd(M) = p$ and $\Ext^i_R(M,R) = 0$ for $i\neq0$ and $i \neq p$.
In fact, the following two observations extend  some results of Auslander from regular rings to hypersurface rings.

\begin{observation}\label{sph}
Let  $(R,\fm)$ be such that its completion is a quotient of equicharacteristic regular local
ring by  a nonzero element and $M$ be torsion-free of constant rank, of  projective dimension $p\in \mathbb{N}$ and locally free. The  following are equivalent:
	\begin{enumerate}
	\item[i)] $\depth(M)+\depth(M^\ast)= \dim R+1$,
	\item[ii)]$M\otimes_RM^*$ is torsion-free,
	\item[iii)] $M$ is $p$-spherical.
\end{enumerate}
\end{observation}

\begin{proof}
$i)\Rightarrow ii)$: Note that $\dim R>0$, because there is  a module of positive projective dimension.  By  Fact \ref{hw}, $\depth(M\otimes_RM^*)>0$. It follows that
  $M\otimes_RM^*$ is $(\Se_1)$. Thus, $M\otimes_RM^*$ is torsion-free.

$ii)\Rightarrow iii)$: Suppose $M\otimes_RM^*$ is torsion-free.
Let $j$ be the smallest positive integer such that  $\Ext^j_R(M,R) \neq0$.
Such a thing exists, because $0<\pd(M)<\infty$. Set $f:R^{\beta_{j}(M)}\to R^{\beta_{j-1}(M)}$. We look at  $L:=\coker(f^\ast)$ and the inclusion $k\subset\Ext^j_R(M,R)\subset L $. This shows that $\depth(L)=0$. Also, there are free modules $F_i$ such that $$0\lo M^\ast\stackrel{f}\lo F_0\lo\ldots\lo F_j\lo L\lo 0\quad(\ast)$$ Since
$\pd(M)<\infty$, $M$ is generically free. Hence, $\Tor_1^R(M,-)$ is torsion. Also, $ \Tor_1^R(M,\coker(f))\subset M\otimes_RM^*$. Thus $\Tor^R_{j+1}(L,M)=\Tor_1^R(M,\coker(f))=0$. By the rigidity theorem of Lichtenbaum \cite[Theorem 3]{l}, $\Tor^R_{i}(L,M)=0$ for all $i>j$. Since
 $\depth(L)=0$ this says that $\pd(M)\leq j$ (see \cite[Proposition 1.1]{au}). By definition, $M$ is $p$-spherical.

 $iii)\Rightarrow i)$:
 Assume that $M$ is $p$-spherical. There is an exact sequence
 $$0\lo M^\ast\lo ( R^{\beta_{0}(M)})^\ast\lo\ldots \lo (R^{\beta_{p}(M)})^\ast\lo L\lo 0.$$ Since $\Ext^p_R(M,R)\subset L $ and $\ell(\Ext^j_R(M,R))<\infty $ we deduce that $\depth (L)=0$. It turns out that
 $\depth(M^\ast)=p+1$. Due to  Auslander-Buchsbaum formula, $\depth(M)+\depth(M^\ast)= \dim R+1$.
\end{proof}

\begin{observation}\label{EI}
	Let  $(R,\fm)$ be as Observation \ref{sph},  $M$ and $N$ be of constant rank, of  finite projective dimension  and be  locally free.  Assume $M\otimes_RN$ is torsion-free.
Then either  $M$ or $N$ is  reflexive.
\end{observation}

\begin{proof} Over zero-dimensional Gorenstein rings any finitely generated module is reflexive.
	Then we may assume that $\dim R>0$. According to Fact \ref{hw} 
$$\depth(M)+\depth(N)\geq \dim R+1.$$ By Auslander-Buchsbaum formula, we may assume that $\depth(N)<\dim R$. From $$\depth(M)+\dim R>\depth(M)+\depth(N)\geq \dim R+1,$$ we conclude
that $\depth(M) \geq  2.$ It turns out that $M$ is $(\Se_2)$ and consequently, $M$ is reflexive.
	\end{proof}

\begin{example}
This is not true that both of $M$ and $N$ are reflexive. Indeed, let $R$ be a two dimensional regular local ring. Let $M:=R$ and $N:=\Syz_1(k)$. The assumptions of the above observation hold. In particular, $M\otimes_RN$ is torsion-free. But, $N$ is not reflexive.
\end{example}

\begin{observation}\label{4}Let  $(R,\fm)$ be   and such that its completion is a quotient of equicharacteristic regular local
	ring by  a nonzero prime element. Let $M$ be of  finite projective dimension. If $(M\otimes M^\ast)^{\bigotimes 2}$ is torsion-free, then $M$ is free. 
\end{observation}

\begin{proof}Since $R$ is a domain, $M$ is of constant rank. Also, the claim in 0-dimensional case follows from the fact that $R$ is domain.
	By induction on $\dim R$, we may assume that $M$ is locally free on the punctured spectrum.
By Observation \ref{EI}, $M\otimes M^\ast$ is reflexive. In view of \cite[Proposition 5.2]{tensor}, $M$ is free.
\end{proof}

 It follows that if $M$ is self-dual and $ M ^{\bigotimes 4}$ is torsion-free, then $M$ is free. 
 In Corollary \ref{4n} we will extend this observation.

\subsection{Being free of relations}

Let $(R,\fm)$ be a regular local ring of dimension $d$.
Auslander proved that the vanishing of $\HH^0_{\fm}(M^{\bigotimes d})=0$ implies freeness of $M$.
It follows easily from \cite[Proposition 3.4(3)]{tensor2} that the vanishing of $\HH^1_{\fm}(M^{\bigotimes (d-1)})=0$ implies freeness of $M$
provided $M$ is locally-free and  torsion-free.  Also, by \cite[Proposition 3.5(3)]{tensor2},   $\HH^2_{\fm}(M^{\bigotimes (d-2)})=0$ implies freeness of $M$
provided $M$ is locally-free and reflexive.

\begin{proposition}
Let $(R,\fm)$ be a regular local ring of dimension $d$ and $\fa$ be an ideal. Let $M$ be  locally-free over $\Spec(R)\setminus\V(\fa)$ and satisfying Serre's condition $(\Se_r)$. If $\HH^r_{\fa}(M^{\bigotimes (d-r)})=0$ then $M$ is free.
\end{proposition}

\begin{proof}First, we point out that $\grade_R(\fa,M)=\inf\{\depth(M_{\fp}):\fp\in\V(\fa)\}\geq r$. In the case $r=d$ we have $\depth(M)=d$. Thanks to Auslander-Buchsbaum formula,  we have $\pd(M)=0$.
Also, if $r=d-1$ then $\HH^r_{\fa}(M)=\HH^{<r}_{\fa}(M)=0$. Hence $ d\geq\depth(M)\geq \grade(\fa, M)=d$. Again,  Auslander-Buchsbaum implies that $\pd(M)=0$.
 Without loss of generality we may assume that $r<d-1$.
Suppose on the contradiction that $\depth(M)<d\quad(+)$.
Recall  from \cite[Lemma 3.7]{acs} that
$$\HH^{r}_{\fa}(M^{\otimes (d-r-1)})=\ldots=\HH^{r}_{\fa}(M^{\otimes 2})=\HH^{r}_{\fa}(M)=0.$$ We apply $\HH^{r}_{\fa}(M\otimes M) =0$ along with \cite[Theorem 3.8]{acs} to deduce that $\Tor_+^R(M,M)=0$ and $\depth_R(M\otimes_R M)>r$. In view of
Fact 3.3.B) $$
\pd_R(M)=\depth_R(M )-\depth_R(M\otimes M).$$
By the same vein,  $\Tor_+^R(M\otimes M,M)=0$ and $\depth_R(M^{\otimes 3})>r$.  In view of
Fact 3.3.B) $\pd_R(M)=\depth_R(M^{\otimes 2})-\depth_R(M^{\otimes 3}).$ Inductively,$$
\pd_R(M)=\depth_R(M^{\otimes j})-\depth_R(M^{\otimes j+1}) \quad(+,+)$$  for all   $1\leq j\leq d-r-1$
and that $\depth_R(M^{\otimes d-r})>r$. We sum all of $d-r-1$ formulas appeared   in $(+,+)$ together to see that
\[\begin{array}{ll}
(d-r-1).\pd_R(M)&=\depth_R(M )-\depth_R(M\otimes M)\\
&\quad+\depth_R(M\otimes M)-\depth_R(M^{\otimes 3})\\
&\quad+\ldots\\
&\quad+\depth_R(M^{\otimes (d-r-1})-\depth_R(M^{\otimes (d-r)})\\
&=\depth_R(M)-\depth_R(M^{\otimes d-r})\\
&\leq (d-1)-\depth_R(M^{\otimes d-r})\\
&< (d-1)-r\\
&=d-r-1\quad(\times)
\end{array}\]
Since
$\pd(M)\stackrel{(+)}\geq 1$ we have $$d-r-1\leq(d-r-1).\pd_R(M)\stackrel{(\times)}< d-r-1.$$ This
 contradiction shows that $\depth(M)=d$, and consequently $M$ is free.
\end{proof}

\begin{lemma}\label{laun}
	Let $R$ be any local  ring, $M$ be locally free  over
	$\Spec(R)\setminus\V(\fa)$  and  $\grade(\fa,M)>0$. If $\HH^0_{\fa}(M\otimes _R M^\ast)= \HH^1_{\fa}(M\otimes _R M^\ast)=0$, then $M$
	is free.
\end{lemma}

\begin{proof}
	Let   $\varphi_M:M\otimes M^\ast\to \Hom (M,M)$ be the natural map.
	Since $M$ is locally free    over
	$\Spec(R)\setminus\V(\fa)$, it follows from   that  $K:=\ker(\varphi_{M})$ and
	$C:=\coker(\varphi_{M})$ are $\fa$-torsion. Since
	$K=\HH^0_{\fa}(K)\subset\HH^0_{\fa}(M\otimes M^\ast)=0$ we have $K=0$. Let $x\in \fa$ be an $M$-sequence. It follows that $x$  is regular over $\Hom (M,M)$, i.e.,  $\grade(\fa,\Hom (M,M))>0$.
Consequently, $$0=\HH^0_{\fa}(\Hom_R(M,M))\to \HH^0_{\fa}(C)\lo\HH^1_{\fa}(M\otimes _R M^\ast)=0.$$
	Thus,  $C=\HH^0_{\fa}(C)=0$. In view of Fact \ref{main5b}.C, $M$
	is free.
\end{proof}
 As another freeness criteria, we show:

\begin{corollary}\label{sph1}
	Let $(R,\fm)$ be a local ring  and  $M$ be locally free   over
	$\Spec(R)\setminus\V(\fa)$   and of finite   projective dimension. If $\grade(\fa,M)+\grade(\fa,M^\ast)\geq \dim R+2$ then $M$ is free.
\end{corollary}

\begin{proof} Since $\grade(\fa,M)+\grade(\fa,M^\ast)\geq \dim R+2$, it follows that $\grade(\fa,M)\geq 2$.
	By Auslander-Buchsbaum formula, $$d:=\dim R\geq\depth(R)\geq\depth(M)\geq\grade(\fa,M)\geq 2.$$
	Let $r:=1$. Then $$\grade(\fa,M)+\grade(\fa,M^\ast)\geq \dim R+r+1$$ and that
	$0<r<d.$
	Due to   Proposition \ref{g} we know that $\HH^0_{\fa}(M\otimes_RM^*)=\HH^1_{\fa}(M\otimes_RM^*)=0$.
	In view of the previous lemma, $M$ is free.
\end{proof}

\begin{example}
	The  assumption $\pd(M)<\infty$ is essential, see Example \ref{gex}(ii). Here, we present another one. Let $(R,\fm,k)$ be any 2-dimensional normal local ring which is not regular.
	Then there is a reflexive module $M$ which is not free, e.g. $M:=\Syz_{2}(k)$.
	Since normality implies $(\Se_2)$ and $(R_1)$, it follows that $M$ is locally free and $\depth(M)=\depth(M^\ast)=\dim R=2$. In particular,
	$\depth(M)+\depth(M^\ast)= \dim R+2$. However, $M$ is not free.
\end{example}

The presented bound in Corollary \ref{sph1} is sharp:
\begin{example}
 Let $(R,\fm,k)$ be any 3-dimensional regular local ring and $M:=\Syz_{2}(k)$.
 Then $M$ is locally free and $\depth(M)+\depth(M^\ast)= 2+2=\dim R+1$. However, $M$ is not free.
\end{example}

Vanishing of $\HH^{1}_{\fm}( M\otimes _RM^\ast)=0$ for any locally free and torsion-free module $M$ over any local
ring $R$ of dimension $d>1$, guarantee that $M$ is free (see \cite[Theorem 8.1]{acs}). This beautiful result extends Lemma \ref{laun} too. Here, we show the dual-twist is important.

\begin{example}\label{41}
	Let  $(R,\fm)$ be a regular local ring of dimension at least $4$. There is a locally free and reflexive module $M$ such that
	\begin{enumerate}
		\item[i)]  $M\otimes M$ is reflexive, i.e., $\HH^{0}_{\fm}( M\otimes _RM)=\HH^{1}_{\fm}( M\otimes _RM)=0$.
		\item[ii)]  $M$ is not free.
	\end{enumerate}
\end{example}

\begin{proof}
	We look at $M:=\Syz_{d-2}(k)$. By \cite[Page 638]{au} we know $M^{\bigotimes (d-1)}$ is torsion-free. Since
	$d-1\geq 3$ and in view of \cite[Proposistion 3.5]{au}, $M\otimes M$ is reflexive.
	So, both claims are clear.
\end{proof}

\section{Depth of tensor powers}
As another application, we are going to compute depth of tensor powers. Our  motivation comes from:

\begin{observation} \label{dpo}Let $(R,\fm)$ be a   local ring of dimension $d$  and  $M$ be locally free  over
	$\Spec(R)\setminus\V(\fa)$. Then $\grade(\fa,M^{\otimes i})\geq d-i\pd(M)$ for all $i>1$.
\end{observation}

\begin{proof}We may assume that $p:=\pd(M)<\infty$.
We argue by induction on $i$.
The case $i=2$ is in the following construction. Now suppose, inductively, that $\grade(\fa,M^{\otimes i})\geq d-ip$.
Let $r:=d-ip-p-1$. Suppose $r<0$. Then  $$\grade(\fa,M^{\otimes i+1})\geq 0\geq r+1=d-(i+1)p,$$as claimed. Without loss of the generality we can assume that
$r\geq0$.
Then $0\leq r<d$ and
$$\grade(\fa,M)+\grade(\fa,M^{\otimes i})\geq(d-p)+(d-ip)=d+r+1.$$ In view of Proposition \ref{g}  we see $$\grade(\fa,M^{\otimes i+1})\geq r+1=d-ip-p=d-(i+1)p,$$as claimed
\end{proof}

Over regular rings, the following result  is due to Huneke-Wiegand (see \cite[Example 3.2]{tensor2}).

\begin{proposition} \label{dpoc}
Let $(R,\fm,k)$ be any  ring and  $M$ be locally free and of projective dimension $1$.  The
following assertions hold:\begin{enumerate}
	\item[i)]  $\depth(M^{\otimes i})= \depth(R)-i$ for all $0<i\leq\depth(R)$  and
	\item[ii)] $\depth(M^{\otimes i})=0$ for all $i\geq  \depth(R)$.
\end{enumerate} 
\end{proposition}

\begin{proof}
i) Set $d:=\depth(R)$.
By induction on $i$  we claim that $\pd(M^{\otimes i})=i<\infty$, e.g., $\depth(M^{\otimes i})= d-i$.
 The case $i=1$ follows by Auslander-Buchsbaum formula. Suppose $i-1<d-1$ and that $\pd(M^{\otimes i-1})= i-1$.
Let $0\to R^n\to R^m\to M\to 0$ be a free resolution. Then we have $$0\lo\Tor_1^R(M,M^{\otimes i-1})\lo R^n\otimes_RM^{\otimes i-1}\lo R^m\otimes_RM^{\otimes i-1}\lo M^{\otimes i}\lo 0.$$
Suppose in the contradiction that $\Tor_1^R(M,M^{\otimes i-1})\neq0$. From locally freeness, $k\subset\Tor_1^R(M,M^{\otimes i-1})$. Thus,
$k\subset\Tor_1^R(M,M^{\otimes i-1})\subset R^n\otimes_RM^{\otimes i-1}$, i.e., $\depth(M^{\otimes i-1})= 0$. But,
$$\depth(M^{\otimes i-1})= d-i+1>0.$$ This contradiction says that $\Tor_1^R(M,M^{\otimes i-1})=0$. Also, $\Tor_{>1}^R(M,M^{\otimes i-1})=0$ because $\pd(M)=1$.
That is the pair $(M,M^{\otimes i-1})$ is Tor-independent. If $P_{\bullet}$ (resp. $Q_{\bullet}$ ) is a minimal free resolution of $M$ (resp. $M^{\otimes i-1}$), then 
 \begin{equation*}
H^n(P_{\bullet}\otimes Q_{\bullet})=\Tor_{n}^R(M,M^{\otimes i-1})=\left\{
\begin{array}{rl}
M^{\otimes i}& \  \   \   \   \   \ \  \   \   \   \   \ \text{if } \   \  n=0  \\
0 & \  \   \   \   \   \ \  \   \   \   \   \ \text{otherwise}\   \
\end{array} \right.
\end{equation*}
Therefore, $P_{\bullet}\otimes Q_{\bullet}$ is a minimal free resolution of $M^{\otimes i}$. From this
$$\pd(M^{\otimes i})=\pd(M^{\otimes i-1})+\pd(M)=(i-1)+1=i.$$
In view of Auslander-Buchsbaum formula,
$$\depth(M^{\otimes i})=\depth(R)-\depth(M^{\otimes i})=d-i.$$

ii) 
By induction on $i$  we claim that $\depth(M^{\otimes d+i})=0$. The case $i=0$ is in part i) where we observed that
 $\depth(M^{\otimes d})=0$. 
 Now suppose, inductively, that $i\geq 1$ and assume the claim for $i-1$. 
Let $$0\lo R^n\lo R^m\lo M\lo 0$$ be a free resolution of $M$. Let $\fp$ be any minimal prime ideal. Note that $R_{\fp}$
is artinian. We localize the sequence at $\fp$ to see that $0\to R_{\fp}^n\to R_{\fp}^m$. Thus,
$$n\ell(R_{\fp})=\ell(R_{\fp}^n)\leq \ell(R_{\fp}^m)=m\ell(R_{\fp}).$$  Consequently,  $n\leq m\quad(+)$. 
We look at $$0\lo\Tor_1^R(M,M^{\otimes d+i-1 })\lo R^n\otimes_RM^{\otimes d+i-1 }\lo R^m\otimes_RM^{\otimes d+i-1 }\lo M^{\otimes d+i}\lo 0 \quad(\ast)$$
By induction hypothesis, we know   $\depth(M^{\otimes d+i-1})=0$. In view of \cite[Proposition 1.1]{au} it follows that $T:=\Tor_1^R(M,M^{\otimes d+i-1  })\neq0$. Clearly, $T$ is of finite length. In  view of Grothendieck's vanishing theorem, $\HH^1_{\fm}(T)=0$.
Also, $\HH^0_{\fm}(T)=T\neq0$. Suppose on the contradiction
that $\HH^0_{\fm}(M^{\otimes d+i})=0$.
We break down
$(\ast)$ into short exact sequences and apply the section functor to deduce the following exact sequences:\begin{enumerate}
	\item[a)]  $0\lo \HH^0_{\fm}(T) \lo \HH^0_{\fm}( R^n\otimes_RM^{\otimes d+i-1 })\lo \HH^0_{\fm}(X)\lo \HH^1_{\fm}(T)=0$,
	\item[b)] $0\lo \HH^0_{\fm}(X)\lo\HH^0_{\fm}( R^m\otimes_RM^{\otimes d+i-1 })\lo \HH^0_{\fm}(M^{\otimes d+i})=0   $.
\end{enumerate}  From the additivity of   length function\[\begin{array}{ll}
n\h^0_{\fm}(  M^{\otimes d+i-1 })&=\h^0_{\fm}( R^n\otimes_RM^{\otimes d+i-1 })\\
&\stackrel{a)}=\h^0_{\fm}(T)+ \h^0_{\fm}(X)\\
&\stackrel{b)}=\ell(T)+\h^0_{\fm}( R^m\otimes_RM^{\otimes d +i-1 })\\
&=\ell(T)+m\h^0_{\fm}( M^{\otimes +i-1 })\\
&> m\h^0_{\fm}( M^{\otimes d+i-1 }).
\end{array}\]

 From this we conclude that $n>m$. This is in a  contradiction with $(+)$. Thus $\HH^0_{\fm}(M^{\otimes d+i})\neq0$. By definition,
 $\depth(M^{\otimes d+i})=0$.  
\end{proof}

\begin{example}
The first item shows that the locally free assumption is important. The second item shows that finiteness of projective dimension is important:
\begin{enumerate}
\item[i)] Let $R$ be a $d$-dimensional Cohen-Macaulay local ring and  let $\underline{x}:=x_1,\ldots, x_{d-1}$ be a parameter sequence and look at $M:=R/\underline{x}R$.
Then $\pd(M)=1$ and that $\depth(M^{\otimes i})=\depth(M)=d-1$ for all $i\geq  1$.
\item[ii)] Let $R:=k[[x,y]]/(xy)$ and let $M:=R/xR$.
Recall that any module over 1-dimensional reduced ring is locally free and that $\depth(M^{\otimes i})=\depth(M)= 1$ for all $i>0$.
\end{enumerate}
\end{example}
Similarly, we have:

\begin{proposition} \label{dpoc2} Let $(R,\fm)$ be any  local ring and $M$ be locally free and of finite projective dimension $p$.  Then $\depth(M^{\otimes i})= \depth(R)-ip$ for all $0<i\leq\frac{\depth(R)}{p}$.
\end{proposition}

\begin{proof} Set $d:=\depth(R)$ and
let $0<i\leq\frac{d}{p}$.
We argue by induction on $i$.  The case $i=1$ is in the Auslander-Buchsbaum formula. Now suppose, inductively, that $i\geq 2$ and assume the claim for $i-1$. This means that
$\depth(M^{\otimes i-1})= d-(i-1)p$.
Let $q$ be the largest number such that  $\Tor_q^R(M, M^{\otimes i-1})\neq0$.  Suppose in the contradiction that $q>0$.
In view of Fact \ref{2au}.A we see $$\depth(M^{\otimes i-1})=\depth(\Tor_q^R(M, M^{\otimes i-1}))+\pd(M)-q=p-q.$$
Since $i\leq d/p$ we have $ip-d\leq0$. Then$$q=p-\depth(M^{\otimes i-1})=p-d+(i-1)p=ip-d\leq0.$$This contradiction says that
$q=0$. Similarly,  $\Tor_+^R(M, M^{\otimes i-2})=0$.  If $P_{\bullet}$ (resp. $Q_{\bullet}$ ) is  a minimal  free resolution of $M$ (resp. $M^{\otimes i-2}$ ), then 
 $P_{\bullet}\otimes Q_{\bullet}$ is a minimal free resolution of $M^{\otimes i-1}$. From this
$\pd(M^{\otimes i-1})$ is finite.
Therefore, in view of Fact \ref{2au}.B) we see
$$\depth(M^{\otimes i})=\depth(M)+\depth(M^{\otimes i-1})-\depth(R)=(d-p)+(d-(i-1)p)-d=d-ip,$$as claimed.
\end{proof}

\begin{example}
Let $(R,\fm)$ be Cohen-Macaulay and let $0\leq i\leq d:=\dim R$. There is a module $M$ such that $\depth(M^{\otimes n})=i$  for all $n\geq  1$.
\end{example}

\begin{proof}
Indeed, let $\underline{x}:=x_1,\ldots, x_{d-i}$ be a parameter sequence and look at $M:=R/\underline{x}R$. Then
$\HH^{<i}_{\fm}(M^{\otimes n})\simeq\HH^{<i}_{\fm}(M)=0$ and $\HH^{i}_{\fm}(M^{\otimes n})\simeq\HH^{i}_{\fm}(M)\neq0$. Thus,
$\depth(M^{\otimes n})=i$  for all $n\geq  1$.
\end{proof}

\begin{observation}
	Let  $(R,\fm)$ be such that its completion is a quotient of equicharacteristic regular local
	ring by  a nonzero element and $M$ be  of constant rank, of finite  projective dimension  and locally free. Assume  in addition that
	$M\simeq M^\ast$. Then  $\depth(M^{\otimes i})$ is constant for all $i>2$.
\end{observation}

\begin{proof}Since $M$ is self-dual it is torsion-free (in fact reflexive). Without loss of the generality we can assume that $\dim R>0$.
There is nothing to prove if $M$ is free. Then, we may assume that $M$ is not free. We are going to show that $\depth(M^{\otimes i})=0$   for all $i>2$. Suppose not, then there
is an $i>2$ such that $\depth(M^{\otimes i})\neq0$. Take such an $i$ in a minimal way. Since $M$ is not free, and in view of Auslander-Buchsbaum formula,
$\depth(M)<d:=\depth(R)$. Recall that $M^{\otimes i}$ is torsion-free, because it is $(\Se_1)$. Let $r:=0$. Then $0\leq r<\dim R$. In particular,
we are in the situation of Fact \ref{hw}. We put things into Fact \ref{hw} to see $$\depth(M^{\otimes i-1})+(d-1)\geq\depth(M^{\otimes i-1})+\depth(M)\geq d+1,$$
and so $$\depth(M^{\otimes i-1})\geq 2\quad(\ast)$$ It follows from the   minimality of $i$  that  $i=3$.
   Due to $(\ast)$, we see $\depth(M^{\otimes 2})\geq 2$. Since $M\simeq M^\ast$, and in view of Lemma \ref{laun} we
see $M$ is free. This is a contradiction that we searched for it. Therefore, $\depth(M^{\otimes i})=0$   for all $i>2$.
\end{proof}

 The above proof shows:
 
\begin{corollary}\label{3}Adopt the above assumption.
 Let $i>2$ and assume in addition that $\dim R>0$. If $M^{\otimes i}$ is torsion-free, then $M$ is free.	
\end{corollary}

\begin{remark}
The assumption $\dim R>0$ is essential. For example, let $R:=k[[x]]/(x^2)$ and look at $M:=k$. For each $i$, we know $M^{\otimes i}=k$ is torsion-free (in fact totally reflexive).
		Clearly, $M$ is not free.
\end{remark}

 The following extends a result of Auslander from regular rings to hypersurfaces:

\begin{corollary}\label{4n}Let  $(R,\fm)$ be  such that its completion is a quotient of equicharacteristic regular local
	ring by  a nonzero prime element. Let $M$ be of  finite projective dimension and be self-dual.
	If $M^{\bigotimes 3}$ is torsion-free, then $M$ is free. 
\end{corollary}

\begin{proof}Note that $R$ is a domain.
By induction on $\dim R$, we may and do assume that $M$ is locally free.	By Corollary \ref{3} $M$ is free. 
\end{proof}

We close the paper by computing $\depth(M^{\otimes n})$ for a module of infinite free resolution.
\begin{example}
Let $(R,\fm,k)$ be any local ring  and let $i\geq2$.  Then \begin{equation*}
\depth(\fm^{\otimes i})=\left\{
\begin{array}{rl}
1& \  \   \   \   \   \ \  \   \   \   \   \ \text{if } \   \  R \textit{ is } \DVR \\
0 & \  \   \   \   \   \ \  \   \   \   \   \ \text{otherwise}\   \
\end{array} \right.
\end{equation*}
The same thing holds for all  $\fm$-primary ideals provided $R$ is a hypersurface ring of dimension bigger than $1$.
\end{example}

\begin{proof}First assume that $\depth(R)=0$. Since $\HH^0_{\fm}(R)\neq 0$, there is a nonzero $r\in R$ such that $r\fm=0$. It follows that $r\in \fm$. Thus, $\HH^0_{\fm}(\fm)\neq 0$, i.e., $\depth(\fm)=0$.
	We   proceed by induction on $i$ that $\depth(\fm^{\otimes i})=0$.  We look at $0\to \fm \to R\to k\to 0$
	and we drive the  exact sequence $$0\lo \Tor_1^R(k,\fm^{\otimes i})\lo \fm^{\otimes i+1}\lo \fm^{\otimes i}\lo\frac{\fm^{\otimes i}}{\fm \fm^{\otimes i}}\lo 0.$$
	Suppose  $\Tor_1^R(k,\fm^{\otimes i})=0$. It follows that $\fm^{\otimes i}$ is free. From this, $\fm^{\otimes i+1}=\oplus \fm $ which is of depth zero.
	Then we may assume that $\Tor_1^R(k,\fm^{\otimes i})\neq0$. It is of finite length.  Thus, $k\subset\Tor_1^R(k,\fm^{\otimes i})$.
	Since $k\subset\Tor_1^R(k,\fm^{\otimes i})\subset \fm^{\otimes i+1}$, we get that
	$\depth(\fm^{\otimes i+1})=0$.

Then without loss of the generality we assume that $\depth(R)>0$.	
In the case $R$ is $\DVR$, the maximal ideal is principal. From this, $\fm$ is free and so $\fm^{\otimes i}$ is free. Thus, $\depth(\fm^{\otimes i})=1$.
Now assume that $R$ is not $\DVR$. In particular, $\beta_2(k)\neq0$. We   proceed by induction
on $i$. When $i=2$  we have $$\tor(\fm^{\otimes 2})=\Tor_2^R(k,k)\simeq k^{\beta_2(k)}.$$
Since $\beta_2(k)\neq0$,   we deduce $\tor(\fm^{\otimes 2})\neq0$. Consequently, $\depth(\fm^{\otimes 2})=0$. Now suppose, inductively, that $\depth(\fm^{\otimes i})=0$. Again,  we drive the  exact sequence $$0\lo \Tor_1^R(k,\fm^{\otimes i})\lo \fm^{\otimes i+1}\lo \fm^{\otimes i}\lo\frac{\fm^{\otimes i}}{\fm \fm^{\otimes i}}\lo 0.$$
Suppose on the contradiction that $\Tor_1^R(k,\fm^{\otimes i})=0$. Then
$\beta_1(\fm^{\otimes i})=0$ and so $\pd(\fm^{\otimes i})=0$. Since $\fm^{\otimes i}$ is free and $R$ is of positive degree we see that $\depth(\fm^{\otimes i})>0$, a contradiction.
This  says that $\Tor_1^R(k,\fm^{\otimes i})\neq0$. It is of finite length.  Thus, $k\subset\Tor_1^R(k,\fm^{\otimes i})$.
Since $$k\subset\Tor_1^R(k,\fm^{\otimes i})\subset \fm^{\otimes i+1},$$ we get that
$\depth(\fm^{\otimes i+1})=0$.

Now assume $I$ is an $\fm$-primary ideal of a hypersurface ring of dimension $d>1$.  We   proceed by induction
on $i>1$ that  $\depth(I^{\otimes i})=0$. Suppose, inductively, that $\depth(I^{\otimes i})=0$ and drive the  exact sequence $0\to \Tor_1^R(R/I,I^{\otimes i})\to I^{\otimes i+1}\to I^{\otimes i}.$
We need to show $\Tor_1^R(R/I,I^{\otimes i})\neq0$.
Suppose on the contradiction that $\Tor_1^R(R/I,I^{\otimes i})=0$.
Due to the first rigidity theorem \cite[2.4]{tensor}, any finite length module over  hypersurface  is rigid. From this, $\Tor_+^R(R/I,I^{\otimes i})=0$ and so $\Tor_+^R(I,I^{\otimes i})=0$. By depth formula over complete-intersection rings (see \cite[2.5]{tensor}) we know that
$$2\leq\depth(I^{\otimes i+1})+\depth(R)=\depth(I)+ \depth(I^{\otimes i})=1+0=1,$$a contradiction.
It remains to check the case $i=2$. This divided in two cases: a) $d>2$ and b) $d=2$.
\begin{enumerate}
	\item[a)]: Let $d>2$. Suppose  $\Tor_1^R(R/I,I)=0$. Then $\Tor_+^R(R/I,I)=0$, and so  $\Tor_+^R(I,I)=0$.
	Hence $$3\leq\depth(I^{\otimes 2})+\depth(R)=\depth(I)+ \depth(I) =2.$$
	This contradiction implies that $\Tor_1^R(R/I,I)\neq0$. Therefore, $\depth(I^{\otimes 2}) =0,$
	because $k\subset\Tor_1^R(R/I,I)\subset I^{\otimes 2}$.
	\item[b)]: Let $d=2$. First assume that $\Tor_1^R(R/I,I)=0$. Recall that
	any finite length module over  hypersurface  is rigid.
	Then $\Tor_+^R(R/I,I)=0$ and so $\Tor_+^R(I,I)=0$.
	Over hypersurfaces, this says that $\pd(I)<\infty$ (see \cite[Theorem 1.9]{tensor2}).
	By Auslander-Buchsbaum formula, $\pd(R/I)=d$. Thus, $\pd(I)=d-1=1$.
	Let $P_{\bullet}$ be a minimal free resolution of $I$. Since  $P_{\bullet}\otimes P_{\bullet}$ is acyclic, 
 we conclude that $P_{\bullet}\otimes P_{\bullet}$ is a minimal free resolution of $I^{\otimes 2}$.
	From this, $\pd(I^{\otimes 2})=2\pd(I)=2$. Thanks to Auslander-Buchsbaum formula,
	$$\depth(I^{\otimes 2})=d-\pd(I^{\otimes 2})=2-2=0.$$
	Then we can assume that $\Tor_1^R(R/I,I)\neq0$.  This implies that $\depth(I^{\otimes 2}) =0.$
\end{enumerate}
The proof is now complete.
\end{proof}

\begin{acknowledgement}
	I would like to thank  Arash Sadeghi and Olgur Celikbas because of a   talk.
\end{acknowledgement}


\begin{thebibliography}{99}

\bibitem{acs} M. Agharzadeh, O. Celikbas,  A.
Sadeghi,  \emph{A study of the cohomological rigidity property}, arXiv:2009.06481 [math.AC]

\bibitem{ag}M. Auslander and O. Goldman, \emph{Maximal orders}, Trans. AMS \textbf{97}
(1960), 1--24.


\bibitem{au} M.
Auslander,  \emph{Modules over unramified regular local rings}, Illinois J. Math. {\bf{5}} (1961) 631--647.


\bibitem{bv}W.
Bruns and U. Vetter, \emph{Length formulas for the local cohomology of exterior powers},  Math.
Z. {\bf{191}} (1986), 145--158.


\bibitem{Da}
~ H. Dao, \emph{Decent intersection and Tor-rigidity for modules
over local hypersurfaces}, Trans. Amer. Math. Soc. {\bf{365}} (2013),
2803--2821.

\bibitem{goto}
S. Goto, \emph{Maximal Buchsbaum modules over regular local rings and a structure theorem for generalized Cohen-Macaulay modules},
Commutative algebra and combinatorics, Adv. Stud. Pure Math. {\bf{11}}, Kinokuniya, Tokyo, North-Holland, Amsterdam (1987),  39--64.

\bibitem{hh}D.
Hanes and C. Huneke, \emph{
Some criteria for the Gorenstein property}, 
J. Pure Appl. Algebra {\bf{201}} (2005), no. 1-3, 4--16.

\bibitem{hsv}
C. Huneke, L.M. Sega and A.N. Vraciu, \emph{Vanishing of Ext and Tor over some Cohen-Macaulay local rings}, Illinois J. Math. {\bf{48}} (2004), no. 1, 295--317. 


\bibitem{tensor}
C. Huneke and R. Wiegand, \emph{Tensor products of modules and the rigidity of Tor}, Math. Ann. {\bf{299}} (1994), 449--476.

\bibitem{tensor2}
C. Huneke and R. Wiegand, \emph{Tensor products of modules, rigidity and local cohomology}, Math. Scand. {\bf{81}} (1997),  161--183.




\bibitem{i}
S.B. Iyengar and R. Takahashi,
\emph{The Jacobian ideal of a commutative ring and annihilators of cohomology},  J. algebra, to appear.

\bibitem{l}
S. Lichtenbaum, \emph{On the vanishing of Tor in regular local rings}, Ill. J. Math. {\bf{10}} (1966),
220--226.

\bibitem{PS2}
C. Peskine and L. Szpiro, \emph{Dimension projective finie et cohomologie locale},
Publ. IHES {{\bf42}}
(1973), 47--119.

\bibitem{bus}J.
St\"{u}ckrad, and W. Vogel,  \emph{Buchsbaum rings and applications. An interaction between algebra, geometry, and topology}, Mathematische Monographien {\bf{21}}, VEB Deutscher Verlag der Wissenschaften, Berlin, 1986.

\bibitem{Wolmer}
Wolmer V. Vasconcelos, \emph{Cohomological degrees and applications},
 Commutative algebra, 667--707, Springer, New York, 2013.

\bibitem{2010}
Wolmer V. Vasconcelos, \emph{Length complexity of tensor products},
Comm. Algebra {\bf{38}} (2010), no. 5, 1743--1760.

\bibitem{yosh2}K.I.
Yoshida,  \emph{Tensor products of perfect modules and maximal
	surjective Buchsbaum modules}, Journal of Pure and Applied Algebra {\bf 123} (1998) 313--326. 


\bibitem{yosh}K.I.
Yoshida,  
\emph{A note on multiplicity of perfect modules of codimension one},
 Comm. Algebra 
 {\bf25} (1997), No. (9) 2807-2816.
\end{thebibliography}
\end{document}